\documentclass{amsart}
\usepackage{amsfonts,amssymb,amsmath,amsthm,mathrsfs}
\usepackage{url}
\usepackage{enumerate}

\def\inn#1#2{\langle#1,#2\rangle}

\newtheorem{theorem}{Theorem}[section]
\newtheorem{lemma}[theorem]{Lemma}

\theoremstyle{definition}

\newtheorem{remark}[theorem]{Remark}
\numberwithin{equation}{section}

\author[G. Hu]{Guoen Hu}
\address{Guoen Hu, Department of Mathematics, School of Science, Zhejiang University of Science and Technology,
	Hangzhou 310023, People's Republic of China}
\email{guoenxx@163.com}
\author[X. Lai]{Xudong Lai}
\address{Xudong Lai, Institute for Advanced Study in Mathematics, Harbin Institute of Technology, Harbin, 150001, People's Republic of China}
\email{xudonglai@hit.edu.cn}
\author[X. Tao]{Xiangxing Tao}
\address{Xiangxing Tao, Department of Mathematics, School of Science, Zhejiang University of Science and Technology,
	Hangzhou 310023, People's Republic of China}
\email{xxtau@163.com}
\author[Q. Xue]{Qingying Xue}
\address{Qingying Xue, School of Mathematical Sciences,
Beijing Normal University,
Laboratory of Mathematics and Complex Systems,	Ministry of Education,
Beijing 100875,
People's Republic of China}
\email{qyxue@bnu.edu.cn}
\thanks{X. Lai was supported by NNSF of China (No. 12322107, No. 12271124 and No. 11801118),
T. Tao   was supported by the NNSF of
China
(No. 12271843) and Zhejiang Province NSF of China (No. LY24A010016), Q. Xue is supported by the National Key R\&D Program of China (No. 2020YFA0712900) and NNSF of China (No. 12271041).}
\thanks{Xiangxing Tao is the corresponding author}

\begin{document}

\title[Calder\'on commutator]{An endpoint estimate for the maximal Calder\'on commutator with rough kernel}

\begin{abstract}In this paper, the authors consider the endpoint estimates for the  maximal Calder\'on
commutator defined  by
$$T_{\Omega,\,a}^*f(x)=\sup_{\epsilon>0}\Big|\int_{|x-y|>\epsilon}\frac{\Omega(x-y)}{|x-y|^{d+1}}
\big(a(x)-a(y)\big)f(y)dy\Big|,$$
where $\Omega$ is homogeneous of degree zero, integrable on $S^{d-1}$ and has vanishing moment of order
one, $a$ be a function on $\mathbb{R}^d$ such that $\nabla a\in L^{\infty}(\mathbb{R}^d)$.
The authors prove that if $\Omega\in L\log L(S^{d-1})$,
then $T^*_{\Omega,\,a}$ satisfies an endpoint estimate of $L\log\log L$ type.
\end{abstract}
\maketitle
\section{Introduction}

As  is well known, the one-dimensional Calder\'on commutator was originated from the seminal
work of Calder\'on \cite{cal2}. In order to represent certain linear differential operators
by means of singular integral operators, Calder\'on  \cite{cal2} introduced the   Calder\'on commutator,
defined as
$$[b,\,H\frac{d}{dx}]f(x)={\rm\,p.\,v.}\int^{\infty}_{-\infty}\frac{b(x)-b(y)}
{(x-y)^2}f(y)dy,$$
where $b$ is a Lipschitz function in $\mathbb{R}$. This operator is a typical non-convolution singular integral
operator, a bilinear operator and plays an important role in the
theory of singular integral operators and have many applications in related topics, for example
see \cite{CM75,Fon16,gra,Mus14I,Mus14II,Mus14III} for details.

The higher-dimensional Calder\'on commutator was also introduced by Calder\'on \cite{cal1}.
Let $\Omega$ be homogeneous of degree zero, integrable on  $S^{d-1}$, the unit sphere in
$\mathbb{R}^d$,  and have vanishing  moment of order one, that is,
\begin{eqnarray}\label{equation1.1}\int_{S^{d-1}}\Omega(\theta)\theta^{\gamma}d\theta=0,\,\,\,\gamma\in\mathbb{Z}_+^d,\,\,|\gamma|=1.\end{eqnarray} Let $a$ be a function on $\mathbb{R}^d$ such that $\nabla a\in L^{\infty}(\mathbb{R}^d)$. Define the  $d$-dimensional Calder\'on commutator $T_{\Omega, a}$ by
\begin{eqnarray}\label{eq:1.2} T_{\Omega, a}f(x)={\rm p. v.}\int_{\mathbb{R}^d}\frac{\Omega(x-y)}{|x-y|^{d+1}}\big(a(x)-a(y)\big)
f(y)dy.\end{eqnarray}
Calder\'on \cite{cal1} proved that if $\Omega\in L\log L(S^{d-1})$, then $T_{\Omega,a}$
is bounded on $L^p(\mathbb{R}^d)$ for all $p\in (1,\,\infty)$.
Ding and the second author \cite{dinglai} considered the
weak type endpoint estimate for $T_{\Omega, a}$, and proved that $\Omega\in L\log L(S^{d-1})$ is a
sufficient condition such that $T_{\Omega,a}$ is bounded from $L^1(\mathbb{R}^d)$ to
$L^{1,\,\infty}(\mathbb{R}^d)$.

To formulate our first result, we recall some facts about the maximal operator associated to the
homogeneous singular integral operator. Let $T_{\Omega}$ be the homogeneous singular integral operator
defined by
\begin{eqnarray}\label{eq1.singular}{T}_{\Omega}f(x)={\rm p.\,v.}\int_{\mathbb{R}^d} \frac
{\Omega( y')}{|y|^d}f(x-y)dy,\end{eqnarray}
where $\Omega$ is homogeneous of degree zero, has mean value zero
and integrable on $S^{d-1}$. $T_{\Omega}^*$, the maximal operator associated with $T_{\Omega}$, is  defined by
\begin{eqnarray}\label{eq1.maximal}T_{\Omega}^*f(x)=\sup_{\epsilon>0}\Big|\int_{|y|>\epsilon}
\frac{\Omega(y)}{|y|^d}f(x-y)dy\Big|.
\end{eqnarray}
It is well known that $\Omega\in L\log L(S^{d-1})$ guarantees that $T_{\Omega}^*$ is bounded
on $L^p(\mathbb{R}^d)$ for all $p\in (1,\,\infty)$. Thanks to the celebrated work of Seeger \cite{seeg}, we know that
$\Omega\in L\log L(S^{d-1})$ is a sufficient condition such that $T_{\Omega}$ is bounded from $L^1(\mathbb{R}^d)$
to $L^{1,\,\infty}(\mathbb{R}^d)$, see also \cite{chr2}.
For a long time, the weak type endpoint estimate for $T_{\Omega}^*$ is an open problem
even for the case of $\Omega\in L^{\infty}(S^{d-1})$. Honz\'{i}k \cite{hon} established a
local endpoint estimate for $T_{\Omega}^*$ when $\Omega\in L^{\infty}(S^{d-1})$.
Bhojak and Mohanty \cite{bhmo} improved the result of Honzik, and proved the following result.
\begin{theorem}\label{dingliweak}
Let $\Omega$ be homogeneous of degree zero, have mean value zero on $S^{d-1}$, and $\Omega\in L\log L
(S^{d-1})$, namely,
$$\int_{S^{d-1}}|\Omega(\theta)|\log ({\rm e}+|\Omega(\theta)|)d\theta<\infty.$$ Then for all $\lambda>0$,
\begin{eqnarray}\label{eq1.weak}|\{x\in\mathbb{R}^d:\,T_{\Omega}^*f(x)>\lambda\}|\lesssim
\int_{\mathbb{R}^d}\Phi_2\big(\frac{|f(x)|}{\lambda}\big)dx,
\end{eqnarray}
where and in the following,  $\Phi_2(t)=t\log\log ({\rm e}^2+t)$.
\end{theorem}

Now we consider the maximal operator associated with the Calder\'on commutator, defined by
$$ T_{\Omega, a}^*f(x)=\sup_{\epsilon>0}\Big|\int_{|x-y|>\epsilon}\frac{\Omega(x-y)}{|x-y|^{d+1}}\big(a(x)-a(y)\big)f(y)dy\Big|.$$
It is well known that $\Omega\in L\log L(S^{d-1})$ is a sufficient condition such that $T_{\Omega,a}$ is bounded on $L^p(\mathbb{R}^d)$ for all $p\in (1,\,\infty)$, see \cite{bc, pwy} and the related references therein. Hofmann \cite{hof1} proved that $\Omega\in L^{\infty}(S^{d-1})$ guarantees the $L^p(\mathbb{R}^d,\,w)$
boundedness of $T_{\Omega,\,a}^*$
for all $p\in (1,\,\infty)$ and $w\in A_p(\mathbb{R}^d)$.
As far as we know, there is no any endpoint estimate for $T_{\Omega,\,a}^*$ when $\Omega$ only satisfies
 size condition, even for the case of $\Omega\in L^{\infty}(S^{d-1})$.
The  purpose of this paper is to establish a counterpart of Theorem \ref{dingliweak}
for  the  operator $T_{\Omega,a}^*$. Our main result can be formulated as follows.

\begin{theorem}\label{dingli1.main}
Let  $\Omega$ be homogeneous of degree zero, satisfy the vanishing condition (\ref{equation1.1}).
Let $a$ be a   function in $\mathbb{R}^d$ such that $\nabla a\in L^{\infty}(\mathbb{R}^d)$.
Suppose that $\Omega\in L\log L(S^{d-1})$,
then for all $\lambda>0$, \begin{eqnarray}\label{eq1.lloglmaximal}
\big|\{x\in\mathbb{R}^d:\,|T_{\Omega,\,a}^*f(x)|>\lambda\}\big|
\lesssim \int_{\mathbb{R}^d}\Phi_2\big(\frac{|f(x)|}{\lambda}\big)dx.
\end{eqnarray}
\end{theorem}

\begin{remark}\label{rem1.3}Let $\eta\in C^{\infty}_0(\mathbb{R}^d)$   such that ${\rm supp}\, \eta\subset\{x:\frac{1}{2}\leq |x|\leq 2\}$ and
$\sum_j\eta_j(x)=1$ for all $x\in \mathbb{R}^d\backslash\{0\}$, where $\eta_j(x)=\eta(2^{-j}x)$.
Set $I_j(x)=\frac{\Omega(x)}{|x|^d}\eta_j(x)$.
The proof of Theorem \ref{dingliweak} \cite{bhmo} is essentially based on the   Fourier transform that
\begin{eqnarray}\label{equation1.fourier}|\widehat{I}_j(\xi)|\lesssim
|2^j\xi|^{-\beta},\,\,\,\hbox{when}\,\,\Omega\in L\log L(S^{d-1}),\end{eqnarray}
where $\beta\in (0,\,1)$ is a constant.
Our strategy to prove Theorem  \ref{dingli1.main} is partly based on the clever
ideas  from \cite{hon,bhmo}.   Since $T_{\Omega,a}$ is not a convolution operator, the
proof in this paper is more
complicated and involves some new ideas and techniques,  see Section 2, Section 3 and Section 4 for
details. The argument in this paper does not involve
(\ref{equation1.fourier}).
\end{remark}

In what follows, $C$ always denotes a
positive constant that is independent of the main parameters
involved but whose value may differ from line to line. We use the
symbol $A\lesssim B$ to denote that there exists a positive constant
$C$ such that $A\le CB$.  Constant with subscript such as $C_1$,
does not change in different occurrences. For any set $E\subset\mathbb{R}^d$,
$\chi_E$ denotes its characteristic function.  For a cube
$Q\subset\mathbb{R}^d$ and $\lambda\in(0,\,\infty)$,
$\lambda Q$  denotes the cube with the same center as $Q$  whose
side length is $\lambda$ times that of $Q$, and $\ell(Q)$ the side length of $Q$.
For a suitable function $f$, we denote $\widehat{f}$ the Fourier transform of $f$.
For $p\in [1,\,\infty]$, $p'$ denotes the dual exponent of $p$, namely, $p'=p/(p-1)$.
For an operator $T$ and a locally integrable function $a$,  $[a,\,T]$ denotes
the commutator of $T$ with symbol $a$, that is,
$$[a,\,T]f(x)=a(x)Tf(x)-T(af)(x).
$$

\section{Approximations to  $T_{\Omega,a}^{**}$}
Let $T_{\Omega,\,a}^{**}$ be the lacunary maximal operator  associated with
 $T_{\Omega,\,a}$, that is,
$$T_{\Omega,a}^{**}f(x)=\sup_{k\in\mathbb{Z}}\Big|
\int_{|x-y|\geq 2^k}\frac{\Omega(x-y)}{|x-y|^{d+1}}\big(a(x)-a(y)\big)f(y)dy\Big|.$$
In this section, we will show that,  when $\Omega$ satisfies some suitable size conditions, $T_{\Omega,a}^{**}$
can be approximated  by   maximal singular integral operators
with smooth kernels. To begin with, we give  a preliminary lemma.
\begin{lemma}\label{lem0.1}Let $\phi\in C^{\infty}_0(\mathbb{R}^d)$  be a radial function such
that ${\rm supp}\, \phi\subset\{1/4\leq |\xi|\leq 4\}$  and
$$\sum_{j\in\mathbb{Z}}\phi^3(2^{-l}\xi)=1,\,\,\,|\xi|>0.$$ Let $S_l$ be the multiplier operator defined by
$$\widehat{S_lf}(\xi)=\phi(2^{-l}\xi)\widehat{f}(\xi).$$
Let $a\in \mathbb{R}^d$ such that $\nabla a\in L^{\infty}(\mathbb{R}^d)$. Then
\begin{eqnarray}\label{equation01}
\Big\|\Big(\sum_{l\in\mathbb{Z}}|2^l[a,\,S_l]f|^2\Big)^{\frac{1}{2}}\Big\|_{L^2(\mathbb{R}^d)}\lesssim \|f\|_{L^2(\mathbb{R}^d)},
\end{eqnarray}
and
\begin{eqnarray}\label{equation02}
\Big\|\sum_{l\in\mathbb{Z}}2^l[a,\,S_l]f_l\Big\|_{L^2(\mathbb{R}^d)}\lesssim
\Big\|\Big(\sum_{l}|f_l|^2\Big)^{1/2}\Big\|_{L^2(\mathbb{R}^d)}.
\end{eqnarray}

\end{lemma}
Inequality (\ref{equation01}) was proved in \cite{chending}, while (\ref{equation02}) follows
from (\ref{equation01}) by a standard duality argument.

The following two lemmas were from \cite{chenhutao}, see also \cite{mahu, chenli} for their  other  versions of special cases.
\begin{lemma}\label{lem0.2} Let $D$, $E$ be positive constants and $E\leq 1$,  $m$ be a  multiplier
such that $m\in L^1(\mathbb{R}^d)$, and
$$\|m\|_{L^{\infty}(\mathbb{R}^d)}\le D^{-1}E$$
and for all multi-indices $\gamma\in\mathbb{Z}_+^d$,
$$\|\partial^{\gamma}m\|_{L^{\infty}(\mathbb{R}^d)}\le D^{|\gamma|-1}.
$$
Let $a$ be a function on $\mathbb{R}^d$ with $\nabla a\in L^{\infty}(\mathbb{R}^d)$, and $T_m$ be the multiplier operator defined by
$$\widehat{T_{m}f}(\xi) = m(\xi)\widehat{f}(\xi).$$
Then for any $\varepsilon\in (0,\,1)$, there exists a constant $C_{\varepsilon}>0$ such that
$$\|[a,\,T_{m}]f\|_{L^2(\mathbb{R}^d)}\le C_{\varepsilon} E^{\varepsilon}\|f\|_{L^2(\mathbb{R}^d)}.
$$
\end{lemma}

\begin{lemma}\label{lem0.3}Let $D$, $A$ and $B$ be positive constants with $A,\,B<1$,  $m$ be a  multiplier
such that $m\in L^1(\mathbb{R}^d)$, and
$$\|m\|_{L^{\infty}(\mathbb{R}^d)}\le D^{-1}(AB)^{2},$$
and for all multi-indices $\gamma\in\mathbb{Z}^d_+$,
$$\|\partial^{\gamma}m\|_{L^{\infty}(\mathbb{R}^d)}\le D^{|\gamma|-1}B^{-|\gamma|}.
$$
Let $T_m$ be the multiplier operator defined by
$$\widehat{T_{m}f}(\xi) = m(\xi)\widehat{f}(\xi).$$
Let $a$ be a function on $\mathbb{R}^d$ such that $\nabla a\in L^{\infty}(\mathbb{R}^d)$. Then for any $\sigma\in (0,\,1)$,
$$\big\|[a,\,T_{m}]f\big\|_{L^2(\mathbb{R}^d)}\lesssim A^{\sigma}B^{\sigma}\|f\|_{L^2(\mathbb{R}^d)}.
$$
\end{lemma}

Let $\eta$ and $\eta_j$ with $j\in\mathbb{Z}$ be the same as in Remark \ref{rem1.3}.
For $j\in\mathbb{Z}$ and $\Omega$ be the function on $S^{d-1}$, let $K_j(x)=\frac{\Omega(x)}{|x|^{d+1}}\eta_j(x)$.
Let  $\psi\in C^{\infty}_0(\mathbb{R}^d)$ be a nonnegative radial function such that
${\rm supp}\, \psi\subset \{x:\,|x|\leq 1/4\}$ and $\int_{\mathbb{R}^d}\psi(x)dx=1$.
For $t\in \mathbb{R}$, set $\psi_t(x) = 2^{dt}\psi(2^{t}x)$. For a positive integer $l$, define
$$K^{l}(x)=\sum_{j\in\mathbb{Z}}K_j*\psi_{l-j}(x).$$
Let $T_{l,a}$ be the operators defined by
\begin{eqnarray}\label{equation2.tl}T_{l,a}f(x)={\rm p.\,v.}\int_{\mathbb{R}^d}K^l(x-y)\big (a(x)-a(y)\big)f(y)dy.\end{eqnarray}
Also, we  define the operators $R_{l,a}$ and $W_{l,a}$ by
$$R_{l,a}f(x)={\rm p.\,v.}\int_{\mathbb{R}^d}\big(K^{l+1}(x-y)-K^l(x-y)\big)(a(x)-a(y))f(y)dy
$$
and
$$W_{l,a}f(x)={\rm p.\,v.}\int_{\mathbb{R}^d}\big(\frac{\Omega(x-y)}{|x-y|^{d+1}}-K^l(x-y)\big)(a(x)-a(y))f(y)dy.
$$

\begin{lemma}\label{lem2.grafa}Let $T$ be an operator which has kernel $K$ in the
sense that, for bounded function $f$ with compact support, and for  $x\in \mathbb{R}^d\backslash {\rm supp}\,f$,
$$Tf(x)=\int_{\mathbb{R}^d}K(x,\,y)f(y)dy.$$
Let $T^*$ be the maximal singular integral operator associated with $T$, defined by
$$T^*f(x)=\sup_{0<\epsilon<R}\Big|\int_{\epsilon<|x-y|\leq R}K(x,\,y)f(y)dy\Big|.
$$
Suppose that $T$ has an extension
which is bounded on $L^2(\mathbb{R}^d)$ with bound $B$, the kernel of $T$ satisfies the size condition that
\begin{eqnarray}\label{equation2.size'}
\sup_{R>0}\int_{R<|x-y|\leq 2R}|K(x,\,y)|dy+\int_{R<|x-y|\leq 2R}|K(x,\,y)|dx\le A_1<\infty,
\end{eqnarray}
and the regularity condition that, for $ y,\,y'\in\mathbb{R}^d$ with $y\not =y'$,
\begin{eqnarray}\label{equation2.regular'}
&&\int_{|x-y|\geq 2|y-y'|}|K(x,y)-K(x,y')|dx\\
&&\quad+\int_{|x-y|\geq 2|y-y'|}|K(y,x)-K(y',x)|dx\lesssim A_2<\infty,\nonumber
\end{eqnarray}
Then $T^*$ is bounded on $L^p(\mathbb{R}^d)$ with bound $C(A_1+A_2+B)\max\{p,\,(p-1)^{-1}\}$, and bounded from $L^1(\mathbb{R}^d)$ to $L^{1,\,\infty}(\mathbb{R}^d)$ with bound
$C(A_1+A_2+B)$.
\end{lemma}

Lemma \ref{lem2.grafa}  is Theorem 1 in \cite{gra0}.

The following theorem is also useful in the proof of Theorem \ref{dingli1.main}.
\begin{theorem}\label{dingli2.9}
Let $\Omega$ be homogeneous of degree zero, satisfies the vanishing moment (\ref{equation1.1})
and $\Omega\in L\log L(S^{d-1})$,
$a$ be a function on $\mathbb{R}^d$ such that $\nabla a\in L^{\infty}(\mathbb{R}^d)$.
Then the lacunary maximal operator associated with $T_{l,a}$, defined by
$$T_{l,a}^{**}f(x)=\sup_{k\in\mathbb{Z}}\Big|\sum_{j>k}\int_{\mathbb{R}^d}K_j*
\psi_{l-j}(x-y)(a(x)-a(y))f(y)dy\Big|,
$$
is bounded  from $L^1(\mathbb{R}^d)$ to $L^{1,\,\infty}(\mathbb{R}^d)$ with bound $C(\|\Omega\|_{L\log L(S^{d-1})}+1)2^l$.
\end{theorem}
\begin{proof}
At first, we claim that if $\Omega\in L^{\infty}(S^{d-1})$ (without the vanishing
moment (\ref{equation1.1})), then
\begin{eqnarray}\label{equation2.3}
\|W_{l,a}f\|_{L^2(\mathbb{R}^d)}\lesssim  \|\Omega\|^*\|f\|_{L^2(\mathbb{R}^d)},
\end{eqnarray}with $$\|\Omega\|^*=\|\Omega\|_{L^1(S^{d-1})}\log(2+\|\Omega\|_{L^{\infty}(S^{d-1})})
+\frac{1}{2+\|\Omega\|_{L^{\infty}(S^{d-1})}}.$$
To prove this,  let $\phi$ be the function in Lemma \ref{lem0.1} and $S_l$ be the operator defined as
$\widehat{S_lf}(\xi)=\phi(2^{-l}\xi)\widehat{f}(\xi).$ Write
\begin{eqnarray*}
W_{l,a}f=\sum_{m\leq 0}\sum_{j\in\mathbb{Z}}[a,S_{m-j}V_{j,l}^mS_{m-j}]f+\sum_{m\ge 1}\sum_{j\in\mathbb{Z}}
[a,\,S_{m-j}V_{j,l}^{m}S_{m-j}]f,
\end{eqnarray*}
where $V_{j,l}^m$   are defined as
$$V_{j,l}^mh(x)=\sigma_{j,l}^m*h(x),$$
with $$\widehat{\sigma_{j,l}^m}(\xi)=\widehat{K_j}(\xi)\big(1-\widehat{\psi}(2^{j-l}\xi)\big)
\phi(2^{j-m}\xi).$$
As it was proved in \cite{duoru}, we have that for some constant $\alpha\in (0,\,1)$,
\begin{eqnarray}\label{eq2.fourier222}|\widehat{K_j}(\xi)|\lesssim 2^{-j}\|\Omega\|_{L^{\infty}(S^{d-1})}|2^j\xi|^{-\alpha},\,\,|\widehat{K_j}(\xi)|
\lesssim 2^{-j}\|\Omega\|_{L^{1}(S^{d-1})}.\end{eqnarray}
On the other hand, we have that $\partial_k\widehat{\psi}(0)=0$  for all $1\leq k\leq d$, and so
\begin{eqnarray}\label{eq2.fourier333}\big|\widehat{\psi}(2^{j-l}\xi)-1|=|\widehat{\psi}(2^{j-l}\xi)-1-\nabla\widehat{\psi}(0)
\cdot 2^{j-l}\xi\big|\lesssim |2^{j-l}\xi|^2.
\end{eqnarray}
It then follows from  (\ref{eq2.fourier222}) and (\ref{eq2.fourier333}) that
\begin{eqnarray}\label{eqn2.four1}\big|\widehat{\sigma_{j,l}^m}(\xi)\big|\lesssim 2^{-j} \min\{1,\,2^{2(m-l)}\}\|\Omega\|_{L^1(S^{d-1})},
\end{eqnarray}
and
\begin{eqnarray}\label{eqn2.four2}\big|\widehat{\sigma_{j,l}^m}(\xi)\big|\lesssim
\|\Omega\|_{L^{\infty}(S^{d-1})}2^{-j}  2^{ -m\alpha}.
\end{eqnarray}
Moreover, for all multi-indices $\gamma\in \mathbb{Z}_+^d$,
\begin{eqnarray}\label{eqn2.four3}|\partial^{\gamma}\widehat{\sigma}_{j,l}^m(\xi)|\lesssim \|\Omega\|_{L^1(S^{d-1})}\Big \{\begin{array}{ll}2^{j(|\gamma|-1)}\,\,&\hbox{if}\,\,m\in\mathbb{N}\\
2^{j(|\gamma|-1)}2^{-|\gamma|m},\,\,&\hbox{if}\,\,m\leq 0.\end{array}
\end{eqnarray}
It now follows from Lemma \ref{lem0.3} (with $A=2^{-l}$, $B=2^m$ and $D=2^j$), (\ref{eqn2.four1}) and (\ref{eqn2.four3}) that when $m\in\mathbb{Z}_-$,
\begin{eqnarray}\label{equation2.26}
\|[a,\,V_{j,l}^{m}]f\|_{L^2(\mathbb{R}^d)}\lesssim \|\Omega\|_{L^1(S^{d-1})} 2^{(m-l)\sigma}
\|f\|_{L^2(\mathbb{R}^d)},
\end{eqnarray}
with $\sigma\in (0,\,1)$. On the other hand, we get by Lemma \ref{lem0.2}, (\ref{eqn2.four2}) and (\ref{eqn2.four3}) that when $m\in\mathbb{Z}_+$,
\begin{eqnarray}\label{eq2.commutator2}
\|[a,\,V_{j,l}^{m}]f\|_{L^2(\mathbb{R}^d)}\lesssim \|\Omega\|_{L^{\infty}(S^{d-1})} 2^{-m\alpha\sigma}
\|f\|_{L^2(\mathbb{R}^d)}.
\end{eqnarray}
Also, we obtain from Lemma \ref{lem0.3}, inequalities (\ref{eqn2.four1}) and (\ref{eqn2.four3}) that when $m\in\mathbb{Z}_+$,
\begin{eqnarray}\label{eq2.commutator3}
&&\|[a,\,V_{j,l}^{m}]f\|_{L^2(\mathbb{R}^d)}\lesssim \|\Omega\|_{L^{1}(S^{d-1})}\|f\|_{L^2(\mathbb{R}^d)}.
\end{eqnarray}

We first consider the case $m\leq 0$. By Lemma \ref{lem0.1}, (\ref{eqn2.four1}) and Plancherel's theorem that
\begin{eqnarray*}
\Big\|\sum_{j\in\mathbb{Z}}[a,S_{m-j}]V_{j,l}^mS_{m-j}f\Big\|^2_{L^2(\mathbb{R}^d)}&\lesssim &
\sum_{j\in\mathbb{Z}}2^{2(j-m)}\|V_{j,l}^{m}S_{m-j}f\|_{L^2(\mathbb{R}^d)}^2\\
&\lesssim& \|\Omega\|_{L^1(S^{d-1})}^22^{2(m-l)}\sum_{j\in\mathbb{Z}}\|S_{m-j}f\|_{L^2(\mathbb{R}^d)}^2\\
&\lesssim&\|\Omega\|_{L^1(S^{d-1})}^22^{2(m-l)}\|f\|_{L^2(\mathbb{R}^d)}^2,
\end{eqnarray*}
and
\begin{eqnarray*}
\Big\|\sum_{j\in\mathbb{Z}}S_{m-j}V_{j,l}^m[a,\,S_{m-j}]f\Big\|^2_{L^2(\mathbb{R}^d)}\lesssim \|\Omega\|_{L^1(S^{d-1})}^22^{2(m-l)}\|f\|_{L^2(\mathbb{R}^d)}^2.
\end{eqnarray*}
Plancherel's theorem, along with inequality (\ref{equation2.26}), leads to that
 \begin{eqnarray*}
\Big\|\sum_{j\in\mathbb{Z}}S_{m-j}[a,\,V_{j,l}^m]S_{m-j}f\Big\|^2_{L^2(\mathbb{R}^d)}\lesssim \|\Omega\|_{L^1(S^{d-1})}^22^{2(m-l)\sigma}\|f\|_{L^2(\mathbb{R}^d)}^2.
\end{eqnarray*}
Therefore,
\begin{eqnarray}\label{equation2.10x}\sum_{m\leq 0}\Big\|\sum_{j\in\mathbb{Z}}
[a,S_{m-j}V_{j,l}^mS_{m-j}]f\Big\|_{L^2(\mathbb{R}^d)}
\lesssim \|\Omega\|_{L^1(S^{d-1})}2^{-\sigma l}\|f\|_{L^2(\mathbb{R}^d)}.\end{eqnarray}

We turn our attention to the case  $m\geq 1$. From Lemma \ref{lem0.1} and (\ref{eqn2.four1}), we deduce that
\begin{eqnarray*}
\Big\|\sum_{j\in\mathbb{Z}}[a,S_{m-j}]V_{j,l}^{m}S_{m-j}f\Big\|^2_{L^2(\mathbb{R}^d)}&\lesssim &
\sum_{j\in\mathbb{Z}}2^{2(j-m)}\|V_{j,l}^{m}S_{m-j}f\|_{L^2(\mathbb{R}^d)}^2\\
&\lesssim&2^{-2m}\|\Omega\|_{L^{1}(S^{d-1})}^2\|f\|_{L^2(\mathbb{R}^d)}^2,
\end{eqnarray*}
and
\begin{eqnarray*}
\Big\|\sum_{j\in\mathbb{Z}}S_{m-j}V_{j,l}^{m}[a,\,S_{m-j}]f\Big\|^2_{L^2(\mathbb{R}^d)}\lesssim
2^{-2m\alpha}\|\Omega\|^2_{L^{\infty}(S^{d-1})}\|f\|_{L^2(\mathbb{R}^d)}^2.
\end{eqnarray*}
On the other hand, we get by (\ref{eq2.commutator3})  that
\begin{eqnarray*}
\Big\|\sum_{j\in\mathbb{Z}}S_{m-j}[a,\,V_{j,l}^{m}]S_{m-j}f\Big\|^2_{L^2(\mathbb{R}^d)}\lesssim
\|\Omega\|_{L^1(S^{d-1})}^2\|f\|_{L^2(\mathbb{R}^d)}^2,
\end{eqnarray*}
and by (\ref{eq2.commutator2}) that
\begin{eqnarray*}
\Big\|\sum_{j\in\mathbb{Z}}S_{m-j}[a,\,V_{j,l}^{m}]S_{m-j}f\Big\|^2_{L^2(\mathbb{R}^d)}\lesssim
2^{-m\alpha\sigma}\|\Omega\|_{L^{\infty}(S^{d-1})}^2\|f\|_{L^2(\mathbb{R}^d)}^2.
\end{eqnarray*}
These  estimates, in turn, imply that
\begin{eqnarray*}\Big\|\sum_{j\in\mathbb{Z}}[a,S_{m-j}V_{j,l}^{i,m}S_{m-j}]f\Big\|_{L^2(\mathbb{R}^d)}
 \lesssim \|\Omega\|_{L^1(S^{d-1})}\|f\|_{L^2(\mathbb{R}^d)},
\end{eqnarray*}
and
\begin{eqnarray*}\Big\|\sum_{j\in\mathbb{Z}}[a,S_{m-j}V_{j,l}^{i,m}S_{m-j}]f\Big\|_{L^2(\mathbb{R}^d)}
\lesssim 2^{-m\alpha\sigma}\|\Omega\|_{L^{\infty}(S^{d-1})}\|f\|_{L^2(\mathbb{R}^d)}.
\end{eqnarray*}
Let $i_0=\lfloor\log (2+\|\Omega\|_{L^{\infty}(S^{d-1})})\rfloor+1,$ and $N=2(\alpha\sigma)^{-1}$. Note that  $$\sum_{m>Ni_0}
2^{-m\alpha\sigma}\lesssim 2^{-Ni_0\alpha\sigma}\lesssim \|\Omega\|^*.
$$ It then follows that
\begin{eqnarray*}
&&\sum_{m\geq 1}\Big\|\sum_{j\in\mathbb{Z}}
[a,\,S_{m-j}V_{j,l}^mS_{m-j}]f\Big\|_{L^2(\mathbb{R}^d)}\\
&&\quad\leq \sum_{1\leq m\leq Ni_0}\Big\|\sum_{j\in\mathbb{Z}}[a,S_{m-j}V_{j,l}^{m}S_{m-j}]f\Big\|_{L^2(\mathbb{R}^d)} \\
&&\qquad+\sum_{m>Ni_0}\Big\|\sum_{j\in\mathbb{Z}}[a,S_{m-j}V_{j,l}^{m}S_{m-j}]f\Big\|_{L^2(\mathbb{R}^d)} \\
&&\quad\lesssim \Big(Ni_0\|\Omega\|_{L^1(S^{d-1})}+\|\Omega\|_{L^{\infty}(S^{d-1})}\sum_{m>Ni_0}
2^{-m\alpha\sigma}\Big)\|f\|_{L^2(\mathbb{R}^d)}\\
&&\quad\lesssim \|\Omega\|^*\|f\|_{L^2(\mathbb{R}^d)}.
\end{eqnarray*}
This, along with (\ref{equation2.10x}), yields (\ref{equation2.3}).

We can now conclude the proof of Theorem \ref{dingli2.9}. Let $L_l(x,\,y)=K^l(x-y)(a(x)-a(y))$.
A trivial computation tells us that
$L_l$ satisfies the size condition  (\ref{equation2.size'}) with $A_1\lesssim C$,
and the regularity condition (\ref{equation2.regular'}) with $A_2\lesssim C2^l$.
Let $\Omega\in L\log L(S^{d-1})$, set $$E_0=\{x'\in S^{d-1}:  |\Omega(x')|\leq 1\}$$
and
$$
E_i=\{x'\in S^{d-1}:\,2^{i-1}<|\Omega(x')|
\leq 2^i\}\,\,\,(i\in\mathbb{N}).
$$
Denote
\begin{eqnarray}\label{equation2.decomposition}\Omega_0(x')=\Omega(x')\chi_{E_0}(x'),\,\,
\Omega_i(x')=\Omega(x')\chi_{E_i}(x')\,\,(i\in\mathbb{N}).\end{eqnarray}
Let $W_{l,\,a}^i$   be defined as $W_{l,\,a}$, but
with $\Omega$ replaced by $\Omega_i$. Our claim  (\ref{equation2.3}) now states that
$$\|T_{\Omega,\,a}f-T_{l,\,a}f\|_{L^2(\mathbb{R}^d)}\lesssim\sum_{i=0}^{\infty}
\|W_{l,\,a}^if\|_{L^2(\mathbb{R}^d)}\lesssim (\|\Omega\|_{L\log L(S^{d-1})}+1)\|f\|_{L^2(\mathbb{R}^d)}.$$
$\Omega\in L\log L(S^{d-1})$,  along with the vanishing condition (\ref{equation1.1}), guarantees
that $T_{\Omega,a}$ is bounded on $L^2(\mathbb{R}^d)$.
Therefore, $T_{l,a}$ is bounded on $L^2(\mathbb{R}^d)$ with bound independent of $l$.
This,
via Lemma \ref{lem2.grafa}, gives us desired conclusion.
\end{proof}

\begin{lemma}\label{thmchenhutao}
Let $A\in (0,\,1)$ be a constant, $\{\mu_j\}_{j\in\mathbb{Z}}$ be a sequence of functions on
$\mathbb{R}^d\backslash \{0\}$. Suppose that for all $j\in\mathbb{Z}$,
$\|\mu_j\|_{L^1(\mathbb{R}^d)}\lesssim 2^{-j}$, and for some $\beta\in (0,\,\infty)$,
\begin{eqnarray}\label{orginfourier}|\widehat{\mu_j}(\xi)|\lesssim 2^{-j}A\min\{|2^j\xi|^{2},\,|2^j\xi|^{-\beta}\},\end{eqnarray}
and for all multi-indices $\gamma\in\mathbb{Z}_+^d$,
$$\|\partial^{\gamma}\widehat{\mu_j}\|_{L^{\infty}(\mathbb{R}^d)}\lesssim 2^{j(|\gamma|-1)}.$$
Let $K(x)=\sum_{j\in\mathbb{Z}}\mu_j(x)$ and $T$ be the convolution operator with kernel $K$. Then for any $\tau\in (0,\,1/2)$,  function $a$ with $\nabla a\in L^{\infty}(\mathbb{R}^d)$,
$$\|[a,\,T]f\|_{L^2(\mathbb{R}^d)}\lesssim A^{\tau}\|f\|_{L^2(\mathbb{R}^d)}.
$$
\end{lemma}
Lemma \ref{thmchenhutao} is a variant of Theorem 3.4 in \cite{chenhutao}, we omit the details for brevity.

\begin{lemma}\label{yinli2.5}
Let $\Omega$ be homogeneous of degree zero (without the vanishing moment (\ref{equation1.1}))
and $\Omega\in L^1(S^{d-1})$,
$a$ be a function on $\mathbb{R}^d$ such that $\nabla a\in L^{\infty}(\mathbb{R}^d)$. Let $R_{l,a}^{**}$ be the
lacunary maximal operator defined by
$$R_{l,a}^{**}f(x)=\sup_{k\in\mathbb{Z}}\Big|\sum_{j>k}\int_{\mathbb{R}^d}
\big(K_j*\psi_{l+1-j}(x-y)-K_j*\psi_{l-j}(x-y)\big)(a(x)-a(y))f(y)dy\Big|.$$
Then
$$\|R^{**}_{l,a}f\|_{L^2(\mathbb{R}^d)}\lesssim 2^l\|\Omega\|_{L^1(S^{d-1})}\|f\|_{L^2(\mathbb{R}^d)}.
$$
\end{lemma}
\begin{proof}
At first, we have that
$$|\widehat{\psi}(2^{j-l}\xi)|\lesssim |2^{j-l}\xi|^{-1},
$$
since $\widehat{\psi}\in \mathscr{S}(\mathbb{R}^d)$. This, together with  (\ref{eq2.fourier333}), leads to  that
\begin{eqnarray}\big|\widehat{\psi_{l-j}}(\xi)-\widehat{\psi_{l+1-j}}(\xi)\big|&=&
\big|\widehat{\psi}(2^{j-l-1})(\xi)-
\widehat{\psi}(2^{j-l})(\xi)\big|\\
&\lesssim&\min\{|2^{j-l-1}\xi|^{-1},\,|2^{j-l}\xi|^2\}\nonumber\\
&\lesssim &2^l\min\{|2^{j}\xi|^2,\,|2^j\xi|^{-1}\},\nonumber
\end{eqnarray}
Thus, by the first inequality in (\ref{eq2.fourier222}), we obtain that
$$|\widehat{K_j}(\xi)\widehat{\psi_{l+1-j}}(\xi)-\widehat{K_j}(\xi)\widehat{\psi_{l-j}}(\xi)|\lesssim 2^l\|\Omega\|_{L^1(S^{d-1})}
2^{-j}\min\{|2^{j}\xi|^2,\,|2^j\xi|^{-1}\}.
$$
On the other hand, we can verify that for all $j\in \mathbb{Z}$ and $\gamma\in\mathbb{Z}_+$,
\begin{eqnarray}\label{eqn2.four13}
|\partial^{\gamma}\big(\widehat{K_j}(\xi)\widehat{\psi}(2^{j-l}\xi)\big)|\lesssim
\|\Omega\|_{L^1(S^{d-1})}2^{j(|\gamma|-1)}
\end{eqnarray}
The last two Fourier transform estimates, via Lemma \ref{thmchenhutao}, gives us that
$$
\|R_{l,\,a}f\|_{L^2(\mathbb{R}^d)}\lesssim2^l\|\Omega\|_{L^1(S^{d-1})}\|f\|_{L^2(\mathbb{R}^d)}.
$$

Observe that the kernel of $R_{l,a}$   satisfies (\ref{equation2.size'}) and
(\ref{equation2.regular'}) with $A_1\lesssim C\|\Omega\|_{L^1(S^{d-1})}$ and $A_2\lesssim Cl\|\Omega\|_{L^1(S^{d-1})}$.
Another application of Lemma \ref{lem2.grafa} leads to our desired conclusion.
\end{proof}

The following approximation of $T_{\Omega,a}^{**}$ when $\Omega\in L^{\infty}(S^{d-1})$ was proved in
\cite{wangzhao}, and
will also useful in the proof of
Theorem \ref{dingli1.main}. Define the operator $[T_{\Omega,a}-T_{l,a}]^{**}$ by
\begin{equation*}
\begin{split}
[T_{\Omega,a}&-T_{l,a}]^{**}f(x)\\
&=\sup_{k\in\mathbb{Z}}\Big|\sum_{j>k}\int_{\mathbb{R}^d}[K_j(x-y)-K_j*
\psi_{l-j}(x-y)](a(x)-a(y))f(y)dy\Big|.
\end{split}
\end{equation*}
\begin{theorem}\label{dingli2.8}
Let $\Omega$ be homogeneous of degree zero (without the vanishing moment (\ref{equation1.1}))
and $\Omega\in L^{\infty}(S^{d-1})$. Then there exists a constant $\tau>0$ such that
$$
\|[T_{\Omega,a}-T_{l,a}]^{**}f\|_{L^2(\mathbb{R}^d)}\lesssim 2^{-\tau l}\|\Omega\|_{L^{\infty}(S^{d-1})}\|f\|_{L^2(\mathbb{R}^d)}.
$$
\end{theorem}
\section{Proof of Theorem \ref{dingli1.main}}

Let $M_{\Omega}$ be the maximal operator defined by
$$M_{\Omega}f(x)=\sup_{r>0}r^{-d}\int_{|x-y|<r}|\Omega(x-y)f(y)|dy.$$
It is well known that $\Omega\in L\log L(S^{d-1})$ guarantees the weak type $(1,\,1)$
boundedness  of $M_{\Omega}$, see \cite{chr2}.
Note that
$$T_{\Omega,\,a}^*f(x)\lesssim M_{\Omega}f(x)+T_{\Omega,\,a}^{**}f(x).$$
Thus, it suffices to prove that for each $\lambda>0$,
\begin{eqnarray}\label{budengshi3.1}\big|\{x\in\mathbb{R}^d:\, T_{\Omega,\,a}^{**}f(x)>\lambda\}\big|
\lesssim \int_{\mathbb{R}^d}\Phi_2\Big(\frac{|f(x)|}{\lambda}\Big) dx.
\end{eqnarray}
By homogeneity, we need only to prove  (\ref{budengshi3.1}) for the case of $\lambda=1$.
For a bounded function $f$, we apply the Calder\'on-Zygmund decomposition to $f$ at level $1$, and
obtain   a collection
of non-overlapping closed dyadic cubes $\mathcal{S}=\{Q\}$, such that $\|f\|_{L^{\infty}(\mathbb{R}^d\backslash \cup_{Q\in\mathcal{S}}Q)}\lesssim 1$,  and
$$\int_{Q} |f(x)| dx \lesssim |Q|,\,\,\,\sum_{Q\in\mathcal{S}}|Q|\lesssim \int_{\mathbb{R}^d} |f(x)|dx.$$
Let  $E=\cup_{Q\in\mathcal{S}}2^{100}dQ$, it is obvious that $|E|\lesssim \int_{\mathbb{R}^d} |f(x)|dx$. We then decompose $f$ as
$f=g+b$, where $$g(x)=f(x)\chi_{\mathbb{R}^d\backslash \cup_{Q\in\mathcal{S}}Q}(x)+f(x)\sum_{Q\in\mathcal{S}}\chi_{Q(|f|\leq 2^{c_1})}(x),
$$
and
$$b(x)=\sum_{Q\in\mathcal{S}}b_Q(x),\,\,b_{Q}(x)=\sum_{l=1}^{\infty}(b_{Q,1}^l(x)+b_{Q,2}^l(x)),$$
with $$b_{Q,\,1}^l(x)=\frac{1}{|Q|}\int_{Q(2^{c_12^{l-1}}<|f|\leq 2^{c_12^l})}f(y)dy\chi_{Q}(x),$$
and
$$b_{Q,2}^l(x)=f(x)\chi_{Q(2^{c_12^{l-1}}<|f|\leq 2^{c_12^l})}(x)-b_{Q,1}^l(x),$$
$c_1>0$ is a constant which will be chosen later.
For each fixed $l\in\mathbb{N}$, set $$G_1^l(x)=\sum_{Q\in\mathcal{S}}b_{Q,1}^l(x),\,\,G_2^l(x)=\sum_{Q\in\mathcal{S}}b_{Q,2}^l(x),\,G^l(x)=G_1^l(x)
+G_2^l(x).$$
For a fixed $j\in\mathbb{Z}$, $\mathcal{S}_j=\{Q:\, Q\in\mathcal{S},\,\ell(Q)=2^j\}$. Let $B_{j}(x)=\sum_{Q\in \mathcal{S}_j}b_Q(x)$ and $B_{j,\,1}^l(x)=\sum_{Q\in \mathcal{S}_j}b_{Q,1}^l(x)$,
$B_{j,2}^l(x)=\sum_{Q\in\mathcal{S}_j}b_{Q,\,2}^l(x)$ for $l\in\mathbb{N}$ and $B_{j}^l(x)=B_{j,1}^l(x)+B_{j,2}^l(x).$
Obviously, $G^l(x)=\sum_{j}B_j^l(x)$, and
\begin{itemize}
\item [\rm (i)] for all cube $Q\in\mathcal{S}$ and $l\in\mathbb{N}$, $\int b_{Q,2}^l(x)dx=0;$
\item[\rm (ii)] for each fixed $l$, $$\|G^l\|_{L^2(\mathbb{R}^d)}^2\lesssim 2^{c_12^l}\|f\|_{L^1(\mathbb{R}^d)};$$
\item [\rm (iii)]\begin{eqnarray}\label{equation3.2}&&\big\|G_1^l\big\|_{L^{\infty}(\mathbb{R}^d)}\lesssim 1,\,\,
\hbox{and}\,\,\big\|\sum_{l=1}^{\infty}G_1^l\big\|_{L^{2}(\mathbb{R}^d)}\leq
\big\|\sum_{l=1}^{\infty}|G_1^l|\big\|_{L^{2}(\mathbb{R}^d)}
\lesssim \|f\|_{L^1(\mathbb{R}^d)}^{\frac{1}{2}}.
\end{eqnarray}
\end{itemize}

By the $L^2(\mathbb{R}^d)$ boundedness of $T_{\Omega,\,a}^{**}$ (ref ???),   we have that
$$\big|\{x\in\mathbb{R}^d:\, |T_{\Omega,\,a}^{**}g(x)|>1/2\}\big|\lesssim \int_{\mathbb{R}^d}|f(x)| dx.$$
Our proof of (\ref{budengshi3.1}) is now reduced to prove that
\begin{eqnarray}\big|\{x\in\mathbb{R}^d\backslash E:\, |T_{\Omega,\,a}^{**}b(x)|>1/2\}\big|\lesssim
\int_{\mathbb{R}^d}\Phi_2(|f(x)|) dx.\end{eqnarray}
To show this estimate, let $E_i$ and $\Omega_i$ ($i\in \mathbb{N}\cup\{0\}$) be the same as in Section 2.
Set $K_j^i(x)=\frac{\Omega_i(x')}{|x|^{d+1}}\eta_j(x)$ and
\begin{eqnarray}\label{eq3.tomega}T_{\Omega,a;\,j}^ih(x)=\int_{\mathbb{R}^d}
K_j^i(x-y)(a(x)-a(y))h(y)dy.\end{eqnarray}
For each fixed $l\in\mathbb{N}$ and $i\in\mathbb{N}\cup\{0\}$, we define $U_{l,a,j}^{i}$  and $U_{l,a,j}$  by
\begin{eqnarray*}
U_{l,a,j}^{i}f(x)=\int_{\mathbb{R}^d}K_j^i*\psi_{2^l-j}(x-y)\big(a(x)-a(y)\big)f(y)dy,
\end{eqnarray*}
and
\begin{eqnarray*}
U_{l,a,j}f(x)=\int_{\mathbb{R}^d}K_j*\psi_{2^l-j}(x-y)\big(a(x)-a(y)\big)f(y)dy
\end{eqnarray*}
respectively. Obviously,
$$U_{l,a,j}f(x)=\sum_{i=0}^{\infty}U_{l,a,j}^if(x).$$

For each fixed $i\in\mathbb{N}$, set $i^*=\lfloor\log(\frac{2}{\tau} (i+2))\rfloor+1$, with $\tau$
the constant appeared in Theorem \ref{dingli2.8}. For $x\in \mathbb{R}^d\backslash E$, let
$${\rm D}_{1}b(x)=\sum_{i=0}^{\infty}
\sum_{l=1}^{\infty}\sup_{k}\Big|\sum_{j>k}(T^i_{\Omega,a,j}-U_{l+i^*,a,j}^i)
G^l(x)
\Big|,$$
$${\rm D}_{2}b(x)=\sum_{i=0}^{\infty}\sup_{k}\Big|\sum_{j>k}\sum_{l=1}^{\infty}(U_{l+i^*,a,j}^i-
U_{i^*,a,j}^i)G_{1}^l(x)\Big|,$$
$${\rm D}_{3}b(x)=\sum_{i=0}^{\infty}\sup_{k}\Big|\sum_{j>k}\sum_{l=1}^{\infty}(U_{i^*,a,j}^i-
U_{0,a,j}^i)G_{1}^l(x)\Big|,$$
$${\rm D}_{4}b(x)=\sup_{k}\Big|\sum_{i=0}^{\infty}\sum_{j>k}U_{0,a,j}^i\big(\sum_{l=1}^{\infty}G^l\big)(x)\Big|=
\sup_{k}\Big|\sum_{j>k}U_{0,a,j}\big(\sum_{l=1}^{\infty}G^l\big)(x)\Big|,$$
$${\rm D}_{5}b(x)=\sum_{i=0}^{\infty}\sum_{l=1}^{\infty}\sup_{k}\Big|\sum_{j>k}(U_{l+i^*,a,j}^i-U_{0,a,j}^i)G_2^l(x)\Big|.$$

It then follows that
$$T_{\Omega,\,a}^{**}b(x)\leq \sum_{m=1}^5D_mb(x).
$$

Consider $D_1b$ firstly. Note that $2^{-\tau2^{l+i^*}}\le 2^{-\tau 2^l}2^{-\tau 2^{i^*}}\lesssim 2^{-\tau 2^l}2^{-2i}$.
An application of Theorem \ref{dingli2.8} yields
\begin{eqnarray*}
\big\|{\rm D}_{1}b\big\|_{L^2(\mathbb{R}^d)}&\lesssim& \sum_{i=0}^{\infty}\sum_{l=1}^{\infty}
2^{-\tau 2^{l+i^*}}\|\Omega_i\|_{L^{\infty}(S^{d-1})}\|G^l\|_{L^2(\mathbb{R}^d)}\\
&\lesssim&\sum_{i=0}^{\infty}\|\Omega_i\|_{L^{\infty}(S^{d-1})}\sum_{l=1}^{\infty}2^{-\tau 2^{l+i^*}}(2^{c_12^l}\|f\|_{L^1(\mathbb{R}^d)})^{1/2}\\
&\lesssim&\|f\|_{L^1(\mathbb{R}^d)}^{\frac{1}{2}},
\end{eqnarray*}
provided that $0\leq {c_1}<\tau$.

Next we consider $D_2b$. Note that
$$|U_{m,a,j}^ih(x)-U_{m-1,a,j}^ih(x)|\leq |U_{m,a,j}^ih(x)-T_{\Omega,a,j}^ih(x)|+|U_{m-1,a,j}^ih(x)
-T_{\Omega,a,j}^ih(x)|.
$$
Theorem \ref{dingli2.8}, along with (\ref{equation3.2}), gives us that
\begin{eqnarray*}
\big\|{\rm D}_{2}b\big\|_{L^2(\mathbb{R}^d)}
&\leq &\sum_{i=0}^{\infty}\big\|\sup_{k\in\mathbb{Z}}\big|\sum_{j>k}\sum_{l=1}^{\infty}\sum_{m=i^*+1}^{l+i^*}(U_{m,a,j}^i-U_{m-1,a,j}^i)
G_{1}^l\big|\big\|_{L^2(\mathbb{R}^d)}\\
&=&\sum_{i=0}^{\infty}\big\|\sup_{k\in\mathbb{Z}}\big|\sum_{j>k}\sum_{m=i^*+1}^{\infty}\sum_{l=m-i^*}^{\infty}(U_{m,a,j}^i-U_{m-1,a,j}^i)
G_{1}^l\big|\big\|_{L^2(\mathbb{R}^d)}\\
&\leq &\sum_{i=0}^{\infty}\sum_{m=i^*+1}^{\infty}\big\|\sup_{k\in\mathbb{Z}}\big|\sum_{j>k}
(U_{m,a,j}^i-U_{m-1,a,j}^i)
\big(\sum_{l=m-i^*}^{\infty}G_{1}^l\big)\big|\big\|_{L^2(\mathbb{R}^d)}\\
&\lesssim&\sum_{i=0}^{\infty}\|\Omega_i\|_{L^{\infty}(S^{d-1})}\sum_{m=i^*+1}^{\infty}
2^{-\tau 2^m} \big\|\sum_{l=1}^{\infty}|G_{1}^l|\big\|_{L^2(\mathbb{R}^d)}\lesssim
\|f\|_{L^1(\mathbb{R}^d)}^{\frac{1}{2}}.
\end{eqnarray*}

Below we consider $D_3b$. By applying  Lemma \ref{yinli2.5}, we obtain that
\begin{eqnarray*}
\big\|{\rm D}_{3}b\big\|_{L^2(\mathbb{R}^d)}
&\leq &\sum_{i=0}^{\infty}\sum_{m=1}^{i^*}\big\|\sup_{k\in\mathbb{Z}}\big|\sum_{j>k}
(U_{m,a,j}^i-U_{m-1,a,j}^i)\big(\sum_{l=1}^{\infty}
G_{1}^l\big)\big|\big\|_{L^2(\mathbb{R}^d)}\\
&\lesssim&\sum_{i=0}^{\infty}(\|\Omega_i\|_{L^1(S^{d-1})})\sum_{m=1}^{i^*}2^m
\big\|\sum_{l=1}^{\infty}G_{1}^l\big\|_{L^2(\mathbb{R}^d)}\\
&\lesssim &\|f\|_{L^1(\mathbb{R}^d)}^{\frac{1}{2}},
\end{eqnarray*}
since $$\sum_{i=0}^{\infty}\|\Omega_i\|_{L^1(S^{d-1})}2^{i^*}\lesssim
\sum_{i=0}^{\infty}\|\Omega_i\|_{L^1(S^{d-1})}(i+2)\lesssim \|\Omega\|_{L\log L(S^{d-1})}.$$

As for ${\rm D}_4b$, we deduce from Theorem \ref{dingli2.9} that
\begin{eqnarray*}
|\{x\in\mathbb{R}^d:\, |{\rm D}_4b(x)|>1/8\}|\lesssim \sum_{l}\|G^l\|_{L^1(\mathbb{R}^d)}\lesssim \|f\|_{L^1(\mathbb{R}^d)}.
\end{eqnarray*}

It remains to estimate ${\rm D}_5b$.
Let
$${\rm D}_{51}b(x)=\sum_{i=0}^{\infty}\sum_{l=1}^{\infty}\sum_{3\leq s\leq N_1(l+i)}
\sup_{k\in\mathbb{Z}}\sum_{j>k}\Big|\sum_{Q\in\mathcal{S}_{j-s}}(U_{l+i^*,a,j}^i-U_{0,a,j}^i)b^l_{Q,2}(x)\Big|,
$$
$${\rm D}_{52}b(x)=\sum_{i=0}^{\infty}\sum_{l=1}^{\infty}\sup_{k\in\mathbb{Z}}\Big|\sum_{j>k}
\sum_{s> N_1(l+i)}(U_{l+i^*,a,j}^i-U_{0,a,j}^i)B_{j-s, 2}^l(x)\Big|,
$$
with $N_1\in\mathbb{N}$, $N_1>d+1$ a constant which will be chosen later. Obviously,
when $x\in\mathbb{R}^d\backslash E$, we have that
${\rm D}_5b(x)\leq {\rm D}_{51}b(x)+{\rm D}_{52}b(x).$
Observe that\begin{eqnarray*}
&&\sum_{l=1}^{\infty}\sum_{3\leq s\leq N_1(l+i)}\sum_j\sum_{Q\in\mathcal{S}_{j-s}}\|b_{Q,\,2}^l
\|_{L^1(\mathbb{R}^d)}\\
&&\quad\lesssim \sum_{l=1}^{\infty}(l+i)\sum_{Q\in\mathcal{S}}\|b_{Q,2}^l\|_{L^1(\mathbb{R}^d)}\\
&&\quad\lesssim i\|f\|_{L^1(\mathbb{R}^d)}+\int_{\mathbb{R}^d}|f(x)|\log\log ({\rm e}^2+|f(x)|)dx.
\end{eqnarray*}
Therefore,
\begin{eqnarray*}
&&\sum_{i=0}^{\infty}\sum_{l=1}^{\infty}\sum_{3\leq s\leq  N_1(l+ i)}\Big\|\sup_{k\in\mathbb{Z}}
\sum_{j>k}\Big|\sum_{Q\in\mathcal{S}_{j-s}}(U_{l+i^*,a,j}^i-U_{0,a,j}^i)b_{Q, 2}^l\Big|
\Big\|_{L^1(\mathbb{R}^d)}\\
&&\quad \le\sum_{i=0}^{\infty}\sum_{l=1}^{\infty}\sum_{3\leq s\leq  N_1(l+i)}
\sum_j\sum_{Q\in\mathcal{S}_{j-s}}\big(\|U_{l+i^*,a,j}^ib^l_{Q,2}\|_{L^1(\mathbb{R}^d)}
+\|U_{0,a,j}^ib^l_{Q,2}\|_{L^1(\mathbb{R}^d)}\big)\\
&&\quad\lesssim \sum_{i=0}^{\infty}\|\Omega_i\|_{L^1(S^{d-1})}\sum_{l=1}^{\infty}
\sum_{3\leq s\leq N_1(l+i)}\sum_j\sum_{Q\in\mathcal{S}_{j-s}}\|b_{Q,2}^l\|_{L^1(\mathbb{R}^d)}\\
&&\quad\lesssim \int_{\mathbb{R}^d}\Phi_2(|f(x)|)dx,
\end{eqnarray*}
This implies  that
$$\big|\{x\in\mathbb{R}^d:\, {\rm D}_{51}b(x)>\frac{1}{16}\}\big|\lesssim \int_{\mathbb{R}^d}\Phi_2(
|f(x)|)dx.$$

The proof of Theorem \ref{dingli1.main} is now reduced to prove that
\begin{eqnarray}\label{eq3.last}\big|\{x\in\mathbb{R}^d:\, {\rm D}_{52}b(x)>\frac{1}{16}\}
\big|\lesssim \int_{\mathbb{R}^d}|f(x)|dx.
\end{eqnarray}
The proof of this estimate is long and complicated, and will  be given in the next section.

\section{Proof of estimate (\ref{eq3.last}) }
For $i\in\mathbb{N}_+$, let $\Omega_i$, $U_{m,a,\,j}^i$ be the same as in Section 3. Set
$$\mathcal{H}_{m,a;\,j}^ih(x)=U_{m,a,j}^ih(x)-U_{m-1,a,j}^ih(x).$$

\subsection{An inequality of Cotlar type}
To prove (\ref{eq3.last}), we need an inequality of Cotlar type. We begin with a preliminary lemma.

\begin{lemma}\label{lem4.1}
Let $j_1,\,j_2,\,j_3\in\mathbb{Z}$. Then
\begin{itemize}
\item[\rm (i)]$$\|\psi_{j_1}*\psi_{j_2} -\psi_{j_1}\|_{L^1(\mathbb{R}^d)}\lesssim 2^{j_1-j_2};$$
\item[\rm (ii)] for $j_2<j_3$, $$\|\psi_{j_1}*(\psi_{j_3} -\psi_{j_2})\|_{L^{\infty}(\mathbb{R}^d)}\lesssim 2^{j_1d}2^{j_1-j_2}.$$
\end{itemize}
\end{lemma}
\begin{proof}
Conclusion (i) was essentially given in \cite[Lemma 2]{hon}.
For the sake of self-contained, we present its proof here.
From the fact that $\int_{\mathbb{R}^d}\psi(y)dy=1$, it follows that
\begin{eqnarray*}
\psi_{j_1}*\psi_{j_2}(x)-\psi_{j_1}(x)=\int_{\mathbb{R}^d}\big(\psi_{j_1}(x-y)-\psi_{j_1}(x)\big)\psi_{j_2}(y)dy
\end{eqnarray*}
Therefore,
$$\|\psi_{j_1}*\psi_{j_2}-\psi_{j_1}\|_{L^1(\mathbb{R}^d)}\leq 2^{j_1}\int_{\mathbb{R}^d}|y||\psi_{j_2}(y)|dy\lesssim 2^{j_1-j_2}.
$$
This, in turn, implies (i).
For conclusion (ii), it is easy to verify that for $x\in\mathbb{R}^d$,
\begin{eqnarray*}|\psi_{j_1}*\psi_{j_3}(x)-\psi_{j_1}(x)|&\leq & \int_{\mathbb{R}^d}
|\psi_{j_1}(x-y)-\psi_{j_1}(x)||\psi_{j_3}(y)|dy\\
&\lesssim &2^{j_1d}\int_{\mathbb{R}^d}2^{j_1}|y||\psi_{j_3}(y)|dy\lesssim 2^{j_1d}2^{j_1-j_3}.
\end{eqnarray*}
Also, we have that for $x\in\mathbb{R}^d$,
\begin{eqnarray*}|\psi_{j_1}*\psi_{j_2}(x)-\psi_{j_1}(x)|\lesssim 2^{j_1d}2^{j_1-j_2}.
\end{eqnarray*}
Combining the estimates above leads to conclusion (ii).
\end{proof}
The following  Cotlar type inequality will be useful in the proof of (\ref{eq3.last}).
\begin{lemma}\label{lem4.2} Let  $\Omega$ be homogeneous of degree zero, and integrable on $S^{d-1}$,
$a$ be a   function in $\mathbb{R}^d$ such that $\nabla a\in L^{\infty}(\mathbb{R}^d)$. For $i\in\mathbb{N}\cup\{0\}$,  $s,\,l,\,m\in\mathbb{N}$ with $s>N_1(l+i)$ and $m\leq l+i^*$, the following Cotlar type  estimate holds
\begin{eqnarray*}
\sup_{k\in\mathbb{Z}}\Big|\sum_{j>k}\mathcal{H}^i_{m,a;j}B_{j-s, 2}^l(x)\Big|&\lesssim  &
\sum_{r=0}^{s2^{l+2}-1}M\Big(\sum_{j\equiv r({\rm mod}\,s2^{l+2})}\mathcal{H}^i_{m,a;\,j}B_{j-s, 2}^l
\Big)(x)\\
&&\qquad+ s2^{l+2}2^{-s}2^{i}\sum_{j\in\mathbb{Z}}
\sum_{Q\in\mathcal{S}_{j-s}}v_{j}*|b_{Q, 2}^l|(x)\\
&&\qquad+s2^{l+2}2^{-s}2^{i}Mb(x),
\end{eqnarray*}
where $v_j(x)=2^{-jd}\chi_{|x|\leq 2^j}(x)$.
\end{lemma}
\begin{proof}
We use some ideas in the proof of  Lemma 2.6 in \cite{bhmo}, but with some more refined
modifications, since $\mathcal{H}^i_{m,a;j}$ is not a convolution operator. Without loss of generality, we may
assume that $\|\nabla a\|_{L^{\infty}(\mathbb{R}^d)}=1.$ It is easy to see that
\begin{eqnarray*}
\sup_{k\in\mathbb{Z}}\Big|\sum_{j>k}\mathcal{H}_{m,a;j}^iB_{j-s,2}^l(x)\Big|\leq \sum_{r=0}^{s2^{l+2}-1}
\sup_{q\in\mathbb{Z}}\Big|\sum_{p\geq q}\sum_{Q\in\mathcal{S}_{ps2^{l+2}+r-s}}
\mathcal{H}_{m,a;ps2^{l+2}+r}^ib_{Q,2}^l(x)\Big|.
\end{eqnarray*}
For each fixed $q$, let $t_q=-qs2^{l+2}-r+s(2^{l+1}+2)$,
and $\Psi_{t_q}$ be the convolution operator with kernel $\psi_{t_q}$ (recall that $\psi_{t_q}(x)=2^{t_qd}\psi(2^{t_q}x)$). For $l\in\mathbb{N}$, we write
\begin{eqnarray*}
&&\sum_{p\geq q}\sum_{Q\in\mathcal{S}_{ps2^{l+2}+r-s}}\mathcal{H}_{m,a;ps2^{l+2}+r}^ib_{Q,2}^l
\\
&&\quad=\sum_{p\geq q}\sum_{Q\in\mathcal{S}_{ps2^{l+2}+r-s}}\mathcal{H}_{m,a;ps2^{l+2}+r}^ib_{Q,2}^l
-\Psi_{t_q}\mathcal{H}_{m,a;ps2^{l+2}+r}^ib_{Q,2}^l\\
&&\quad+\sum_{p\in\mathbb{Z}}\sum_{Q\in\mathcal{S}_{ps2^{l+2}+r-s}}
\Psi_{t_q}\mathcal{H}_{m,a;ps2^{l+2}+r}^ib_{Q,2}^l\\
&&\quad+\sum_{p< q}\sum_{Q\in\mathcal{S}_{ps2^{l+2}+r-s}}
\Psi_{t_q}\mathcal{H}_{m,a;ps2^{l+2}+r}^ib_{Q,2}^l\\
&&:={\rm I}_q(x)+{\rm II}_q(x)+{\rm III}_q(x).
\end{eqnarray*}
For each $q\in \mathbb{Z}$, it is obvious that
\begin{eqnarray}\label{equation4.1iiq}|{\rm II}_q(x)|\lesssim M\Big(\sum_{p\in\mathbb{Z}}\mathcal{H}_{m, a,ps2^{l+2}+r}^i
B_{ps2^{l+2}+r-s,2}^l\Big)(x).\end{eqnarray}

We now estimate the term $\sup_q|{\rm I}_q|$. Let $T_{\Omega,j}$ be the convolution operator with kernel $K_j$, and
$$T^i_{\Omega,a;\,j}h(x)=\int_{\mathbb{R}^d}K_j^i(x-y)(a(x)-a(y))h(y)dy.
$$Set
$$\widetilde{\psi}_{2^m-ps2^{l+2}-r}(y)=\psi_{2^m-ps2^{l+2}-r}(y)-\psi_{2^{m-1}-ps2^{l+2}-r}(y).$$
Let $\Psi_{t_q}$, $\widetilde{\Psi}_{2^m-ps2^{l+2}-r}$ be the convolution operators with kernels
$\psi_{t_q}$
and
$\widetilde{\psi}_{2^m-ps2^{l+2}-r}$
respectively. For each fixed $p\geq q$ and $Q\in\mathcal{S}$ with $\ell(Q)=2^{ps2^{l+2}+m-s}$, write
\begin{eqnarray*}
&&\mathcal{H}_{m,a;ps2^{l+2}+r}^ib_{Q,2}^l(x)-\Psi_{t_q}\mathcal{H}_{m,a;ps2^{l+2}+r}^ib_{Q,2}^l(x)\\
&&\quad=(\Psi_{2^m-ps2^{l+2}-r}-\Psi_{t_q}\Psi_{2^{m}-ps2^{l+2}-r})T^i_{\Omega,a,ps2^{l+2}+r}b_{Q,2}^l(x)\\
&&\qquad-\big(\Psi_{2^{m-1}-ps2^{l+2}-r}-\Psi_{t_q}\Psi_{2^{m-1}-ps2^{l+2}-r}\big)
T^i_{\Omega,a,\,ps2^{l+2}+r}b_{Q,2}^l(x)\\
&&\qquad+\big[a,\,\Psi_{2^m-ps2^{l+2}-r}-\Psi_{t_q}\Psi_{2^{m}-ps2^{l+2}-r}\big]
T^i_{\Omega,ps2^{l+2}+r}b_Q^l(x)\\
&&\qquad-\big[a,\Psi_{2^{m-1}-ps2^{l+2}-r}-\Psi_{t_q}\Psi_{2^{m-1}-
ps2^{l+2}-r}\big]T^i_{\Omega,ps2^{l+2}+r}b_{Q,2}^l(x)\\
&&\qquad+[a,\Psi_{t_q}]\big(\Psi_{2^{m}-ps2^{l+2}-r}-\Psi_{2^{m-1}-ps2^{l+2}-r}\big)
T^i_{\Omega,ps2^{l+2}+r}b_{Q,2}^l(x)\\
&&\quad:=\sum_{t=1}^5{\rm D}_t(p,q,Q)(x).
\end{eqnarray*}
Note that when $p\geq q$,
$${\rm supp}\big(\psi_{2^m-ps2^{l+2}-r}-\psi_{t_q}*\psi_{2^m-ps2^{l+2}-r}\big)
\subset\{x:\,|x|\leq 2^{ps2^{l+2}+r+1}\},
$$
it then follows from (i) of Lemma  \ref{lem4.1} that
\begin{eqnarray*}
&&\int_{\mathbb{R}^d}\big|K^i_{ps2^{l+2}+r}(x-y-z)\big(a(x)-a(y+z))\\
&&\quad\qquad\times \big(\psi_{2^m-ps2^{l+2}-r}(y)-
\psi_{t_q}*\psi_{2^m-ps2^{l+2}-r}(y)\big)\big|dy\\
&&\quad\lesssim  2^{-d(ps2^{l+2}+r)}\chi_{\{|x-z|\leq 2^{ps2^{l+2}+r+2}\}}(x-z)2^{-t_q+2^m-ps2^{l+2}-r}2^i,
\end{eqnarray*}
since $\|\Omega_i\|_{L^{\infty}(S^{d-1})}\leq 2^i$. The facts that $s>2i$ and $m\leq l+i^*$  implies that
$$-t_q+2^m-ps2^{l+2}-r\leq -s(2^{l+2}+2)+i2^l<-2s-s2^{l+1},$$
when $p\geq q$. Therefore,
\begin{eqnarray*}\sup_{q}\sum_{p\geq q}\sum_{Q\in \mathcal{S}_{ps2^{l+2}+r-s}}|{\rm D}_1(p,q,Q)(x)|\lesssim
2^{i}2^{-s}\sum_{j}\sum_{Q\in \mathcal{S}_{j-s}}v_{j}*(|b_{Q,2}^l|)(x).
\end{eqnarray*}
Similarly, we get that
\begin{eqnarray*}
\sup_{q}\sum_{p\geq q}\sum_{Q\in \mathcal{S}_{ps2^{l+2}+r-s}}|{\rm D}_2(p,q,Q)(x)|
\lesssim 2^{i}2^{-s} \sum_{j}\sum_{Q\in \mathcal{S}_{j-s}}v_{j}*(|b_{Q,2}^l|)(x).
\end{eqnarray*}
Also, we have by (i) of Lemma \ref{lem4.1} that
\begin{eqnarray*}
&&\int_{\mathbb{R}^d}\big|K^i_{ps2^{l+2}+r}(y-z)\big(a(x)-a(y))\\
&&\quad\qquad\times \big(\psi_{2^m-ps2^{l+2}-r}(x-y)-
\psi_{t_q}*\psi_{2^m-ps2^{l+2}-r}(x-y)\big)\big|dy\\
&&\quad\lesssim  2^{-d(ps2^{l+2}+r)}\chi_{\{|x-z|\leq 2^{ps2^{l+2}+r+2}\}}(x-z)2^{-t_q}2^i.
\end{eqnarray*}
Therefore,
\begin{eqnarray*}
&&\sup_{q}\sum_{p\geq q}\sum_{Q\in \mathcal{S}_{ps2^{l+2}+r-s}}
|{\rm D}_3(p,q,Q)(x)|
\lesssim 2^{i}2^{-s}\sum_{j}\sum_{Q\in \mathcal{S}_{j-s}}v_{j}*(|b_{Q,2}^l|)(x).
\end{eqnarray*}
and similarly,
\begin{eqnarray*}
&&\sup_{q}\sum_{p\geq q}\sum_{Q\in \mathcal{S}_{ps2^{l+2}+r-s}}|{\rm D}_4(p,q,Q)(x)|
\lesssim 2^{i}2^{-s}\sum_{j}\sum_{Q\in \mathcal{S}_{j-s}}v_{j}*(|b_{Q,2}^l|)(x).
\end{eqnarray*}
Observing that when $p\geq q$,
\begin{eqnarray*}
\int_{\mathbb{R}^d}\big|\psi_{t_q}(x-y)\widetilde{\psi}_{2^m-ps 2^{l+2}-r}(y-z)\big|dy\lesssim
2^{d(2^m-ps 2^{l+2}-r)}
\chi_{\{|z-x|\leq C2^{ps 2^{l+2}+r}\}}(z),
\end{eqnarray*}
we then have that
\begin{eqnarray*}
&&\Big|[a,\Psi_{t_q }]\big(\Psi_{2^{m}-ps2^{l+2}-r}-\Psi_{2^{m-1}-ps2^{l+2}-r}\big)
T^i_{\Omega,ps2^{l+2}+r}b_{Q,2}^l(x)\Big|\\
&&\quad=2^{-ps2^{l+2}-r-t_q}2^{d(2^m-ps 2^{l+2}-r)}\nonumber\\
&&\qquad\times\int_{|x-z|\leq C2^{ps 2^{l+2}+r}}2^{ps2^{l+2}+r}\big|
T^i_{\Omega,ps2^{l+2}+r}b_{Q,2}^l(z)\big|dz\nonumber\\
&&\quad\lesssim2^{-s2^{l+1}-2s}2^i v_{ps2^{l+2}+r}*(|b_{Q,2}^l|)(x).\nonumber
\end{eqnarray*}
Therefore,
\begin{eqnarray}\label{equation4.1iq}\sup_{q}|{\rm I}_q(x)|&\lesssim &2^{-s}2^{i}\sum_{j\in\mathbb{Z}}
\sum_{Q\in\mathcal{S}_{j-s}}
v_j*(|b_{Q,2}^l|) (x).
\end{eqnarray}

We turn our attention to term ${\rm III}_q$. Let
\begin{eqnarray*}L(x,\,z)&=&\int_{\mathbb{R}^d}\psi_{t_q}(x-y)K_{ps2^{l+2}+r}^i*\widetilde{\psi}_{2^m-ps2^{l+2}-r}(y-z)
\big(a(y)-a(z))dy\\
&=&\int_{\mathbb{R}^d}\psi_{t_q}(x-y-z)K_{ps2^{l+2}+r}^i*\widetilde{\psi}_{2^m-ps2^{l+2}-r}(y)
\big(a(y+z)-a(z))dy.
\end{eqnarray*}
For each cube $Q\in \mathcal{S}_{ps2^{l+2}+r-s}$, let $z_Q$ be the center of $Q$.
Using the cancelation of $b_{Q,\,2}^l$, we can write
\begin{eqnarray*}
\Psi_{t_q}\mathcal{H}_{m,a;ps2^{l+2}+r}^ib_{Q,2}^l(x)&=&\int_{\mathbb{R}^d}
L(x,\,z)b_{Q,\,2}^l(z)dz\\
&=&
\int_{\mathbb{R}^d}\big[L(x,\,z)-L(x,\,z_Q)\big]b_{Q,2}^l(z)dz.
\end{eqnarray*}
Now set
$${\rm F}_1(x,z,z_Q)=\int_{\mathbb{R}^d}\psi_{t_q}(x-y-z)K_{ps2^{l+2}+r}^i*\widetilde{\psi}_{2^m-ps2^{l+2}-r}(y)
\big(a(z_Q)-a(z))dy,
$$
$${\rm F}_2(x,z,z_Q)=\int_{\mathbb{R}^d}\psi_{t_q}(x-y-z)K_{ps2^{l+2}+r}^i*\widetilde{\psi}_{2^m-ps2^{l+2}-r}(y)
\big(a(y+z)-a(y+z_Q))dy,$$
and
\begin{eqnarray*}{\rm F}_3(x,z,z_Q)&=&\int_{\mathbb{R}^d}\big(\psi_{t_q}(x-y-z)-\psi_{t_q}(x-y-z_Q)\big)\\
&&\times
K_{ps2^{l+2}+r}^i*\widetilde{\psi}_{2^m-ps2^{l+2}-r}(y)
\big(a(y+z_Q)-a(z_Q))dy.
\end{eqnarray*}
We then  decompose  $L(x,z)-L(x,\,z_Q)$ as
\begin{eqnarray*}
L(x,z)-L(x,\,z_Q)={\rm F}_1(x,z,z_Q)+{\rm F}_2(x,z,z_Q)+{\rm F}_3(x,z,z_Q).
\end{eqnarray*}
An application of Young's inequality gives us that
\begin{eqnarray*}\|K_{ps2^{l+2}+r}^i*\widetilde{\psi}_{2^m-ps2^{l+2}-r}\|_{L^{\infty}(\mathbb{R}^d)}
&\leq& \|K_{ps2^{l+2}+r}^i\|_{L^{\infty}(\mathbb{R}^d)}
\|\widetilde{\psi}_{2^m-ps2^{l+2}-r}\|_{L^1(\mathbb{R}^d)}\\
&\lesssim &2^i2^{-(ps2^{l+2}+r)(d+1)}.
\end{eqnarray*}
Therefore,
\begin{eqnarray*}|{\rm F}_1(x,z,z_Q)|&\lesssim &2^i2^{t_qd-(ps2^{l+2}+r)d}2^{-s}
\int_{\{|y|\leq 2^{ps2^{l+2}+r+2}\}}dy\chi_{\{|x-z|\leq C2^{-t_q}\}}(z)\\
&\lesssim& 2^i2^{-s}2^{t_qd}\chi_{\{|x-z|\leq C2^{-t_q}\}}(z),
\end{eqnarray*}
since $t_q -ps2^{l+2}-r<0$.
Similarly, we have that
$$|{\rm F}_2(x,z,z_Q)|\lesssim 2^i2^{-s}2^{t_qd}\chi_{\{|x-z|\leq C2^{-t_q}\}}(z).
$$
As for term ${\rm F}_3$, a trivial computation involving the mean vale theorem leads to that
\begin{eqnarray*}
|{\rm F}_3(x,z,z_Q)|&\lesssim&2^{t_q(d+1)}|z-z_Q|
2^i2^{-(ps2^{l+2}+r)(d+1)}\\
&&\quad\times \int_{\{|y|\leq C2^{ps2^{l+2}+r}\}}|y|dy\chi_{\{|x-z|\leq C2^{-t_q}\}}(z)\\
&\lesssim &2^i2^{-s}2^{t_qd}\chi_{\{|x-z|\leq C2^{-t_q}\}}(z).
\end{eqnarray*}
Combining the estimates for term ${\rm F}_1$,  ${\rm F}_2$ and ${\rm F}_3$ yields
$$|{\rm III}_q(x)|\lesssim 2^i2^{-s}Mb(x).
$$
This, along with estimates (\ref{equation4.1iiq}) and (\ref{equation4.1iq})
then  completes the proof of Lemma \ref{lem4.2}.
\end{proof}

\subsection{Micro-local decomposition for $\mathcal{H}^i_{m,a;j}$}
As is well known, the micro-local decomposition, introduced by Seeger \cite{seeg}, plays an
important role in
the proof of the endpoint estimates for rough operators. Ding and the second author \cite{dinglai}
generalized Seeger's estimates in \cite{seeg} to certain
non-convolution operators, from which they established the weak type endpoint estimates
for Calder\'on commutator. To prove (\ref{eq3.last}), we need  some
ideas from \cite{seeg, dinglai}, together with some more refined micro-local decomposition estimate for
$\mathcal{H}^i_{m,a;j}$. Recall that
$$K_j^i(x)=\frac{\Omega_i(x)}{|x|^{d+1}}\eta_j(|x|),$$
and
$$\mathcal{H}_{m,a;\,j}^ih(x)=\int_{\mathbb{R}^d}\widetilde{K}_{j,m}^i(x-y)(a(x)-a(y))h(y)dy,$$
with
$$\widetilde{K}_{j,m}^i(x)=K_j^i*\widetilde{\psi}_{2^m-j}(x),\,\,\widetilde{\psi}_{2^m-j}(x)
=\psi_{2^{m}-j}(x)-\psi_{2^{m-1}-j}(x).$$

Let  $0<\iota<1$ be a constant which will be chosen later.
For $s\in\mathbb{N}$, we define $R^j_{s,a}$ as
	$$
		R^j_{s,a}(x,y)=\int_{\mathbb{R}^d}\psi_{\iota s+3-j}(x-z)(a(z)-a(y))dz.
	$$
We then have that
\begin{eqnarray}
&&|R^j_{s,a}(x,y)|=\Big|\int_{\mathbb{R}^d}\psi_{\iota s+3-j}(z)(a(x-z)-a(y))dz\Big|
\lesssim 2^j,\,\,\hbox{if}\,\,|x-y|\lesssim 2^j.
\end{eqnarray}
Define the operator $\mathcal{J}_{m,\,a;\,j}^i$ as
$$\mathcal{J}_{m,\,a;\,j}^ih(x)=\int_{\mathbb{R}^d}\widetilde{K}_{j,m}^i(x-y)R^j_{s,\,a}(x,\,y)h(y)dy.$$
Observe that
$$|R^j_{s,a}(x,y)-(a(x)-a(y))|\leq \int_{\mathbb{R}^d}\psi_{\iota s+3-j}(z)|a(x)-a(x-z)|dz
\lesssim2^{j-\iota s},$$
and
$$\|K_j^i*\psi_{2^m-j}-K_j^i*\psi_{2^{m-1}-j}\|_{L^1(\mathbb{R}^d)}\lesssim \|\Omega_i\|_{L^1(S^{d-1})}2^{-j}.
$$
It then follows that
\begin{eqnarray}\label{eq4.approximation}
\|\mathcal{J}_{m,\,a;\,j}^if-\mathcal{H}_{m,a;\,j}^if\|_{L^1(\mathbb{R}^d)}\lesssim 2^{-\iota s}\|\Omega_i\|_{L^1(S^{d-1})}
\|f\|_{L^1(\mathbb{R}^d)}.
\end{eqnarray}

For each cube $Q\in\mathcal{S}_{j-s}$,
we split $\mathcal{J}_{m,\,a;\,j}^ib^l_{Q,2}$ as $$\mathcal{J}_{m,\,a;\,j}^ib^l_{Q,2}=\Psi_{s\kappa-j}
\mathcal{J}_{m,\,a;\,j}^ib^l_{Q,2}+(I-\Psi_{s\kappa-j})\mathcal{J}_{m,\,a;\,j}^ib^l_{Q,2},$$ where
$\Psi_t$ is the convolution operator defined by
$\Psi_tf(x)=\psi_{t}*f(x)$,	
$\kappa \in (0,1)$ is a constant which will   be chosen later.
\begin{lemma}\label{lemma5.3}Under the hypothesis of Theorem \ref{dingli1.main}, we have that for
each $j\in\mathbb{Z}$, $s\geq 3$ and
$Q\in \mathcal{S}_{j-s}$, $$\big \| \Psi_{s\kappa -j} \mathcal{J}_{m,\,a;\,j}^ib^l_{Q,2}\big\|_{L^1(\mathbb{R}^d)}
\lesssim \big(2^{-(1-\kappa)s} +2^{-(1-\iota)s} \big) \|\Omega_i\|_{L^{\infty}({S}^{d-1})}
\|b^l_{Q,2}\|_{L^1(\mathbb{R}^d)}.$$	
\end{lemma}
\begin{proof}Write
\begin{eqnarray*}\Psi_{s\kappa-j}\mathcal{J}_{m,\,a;\,j}^ib_{Q,2}^l(x)=
\int_{\mathbb{R}^d}\int_{\mathbb{R}^d}\psi_{s\kappa-j}(x-y)\widetilde{K}^i_{j,m}(y-z)
R_{s,a}^{j}(y,z)b^l_{Q,2}(z)dz dy.
\end{eqnarray*}
A straightforward computation involving the vanishing moment of $b_{Q,2}^l$ leads to that
\begin{eqnarray*}
&&|\Psi_{s\kappa-j}\mathcal{J}_{m,\,a;\,j}^ib^l_{Q,2}(x)|
\leq\inf_{y'\in{Q}}\int_{\mathbb{R}^d}\int_{S^{d-1}}|\Omega_i(\theta)|
\Big|\int_{\mathbb{R}^d}\int_{0}^{\infty}(\psi_{s\kappa-j}(x-y-t-r{\theta})\\
&&\quad\times R_{s,a}^{j}(y+t+r{\theta},y)-
\psi_{s\kappa-j
}(x-y'-t-r{\theta})R_{s,a}^{j}(y'+t+r{\theta},y'))\\
&&\qquad\frac{\eta (2^{-j}r)}{r^2}drb^l_{Q,2} (y)dy\Big|dy'd\theta\widetilde{\psi}_{2^m-j}(t)dt
\end{eqnarray*}

Set
\begin{eqnarray*}{\rm U}(x)&=&\frac{1}{|Q|}\int_{\mathbb{R}^d}\int_{S^{d-1}}|\Omega_i(\theta)|
\int_{Q}\Big|
\int_{\mathbb{R}^d}\int_{2^{j-1}}^{2^j}\Big(\psi_{s\kappa-j}(x-y-t-r\theta)\\
&-&\psi_{s\kappa-j}(x-y'-t-r\theta)\Big)R_{s,a}^{j}(y+t+r\theta,y){\frac{dr}{r^2}}b^l_{Q,2}(y)dy
\Big|dy'd\theta \widetilde{\psi}_{2^m-j}(t)dt
\end{eqnarray*}
and
\begin{eqnarray*}&&{\rm V}(x)=\frac{1}{|Q|}\int_{\mathbb{R}^d}\int_{S^{d-1}}|\Omega_i(\theta)|
\int_{Q}\Big|\int_{\mathbb{R}^d}
\int_{2^{j-1}}^{2^j}\psi_{s\kappa-j}(x-y'-t-r\theta) \\
&&\quad\times \Big(R_{s,a}^{j}(y+t+r\theta,y)-R_{s,a}^{j}(y'+t+r\theta,y')\Big){\frac{dr}{r^2}}b^l_{Q,2}(y)dy\Big|dy'd
\theta \widetilde{\psi}_{2^m-j}(t)dt.
\end{eqnarray*}
It then follows that
$$|\Psi_{s\kappa-j}\mathcal{J}_{m,\,a;\,j}^ib^l_{Q,2}(x)|\leq {\rm U}(x)+{\rm V}(x).$$
Note that
\begin{eqnarray*}|R_{s,a}^j(y+t+r\theta,y)|&\leq&\int_{\mathbb{R}^d}|\psi_{\iota s+3-j}(y+t+r\theta-z)
(a(z)-a(y))|dz\\
&\lesssim &2^{j-\iota s-3}+r+|t|.
\end{eqnarray*}A trivial computation leads to that
\begin{eqnarray*}
\|{\rm U}\|_{L^1(\mathbb{R}^d)}&\lesssim&\int_{\mathbb{R}^d}\int_{S^{d-1}}|\Omega_i(\theta)|
\int_{\mathbb{R}^d}
\int_{2^{j-1}}^{2^j}2^{-j+s\kappa}\|\nabla{\psi}\|_{L^1}2^{j-s}(2^{j-\iota s-3}+r+|t|){\frac{dr}{r^2}}\\
&&\quad\times|b^l_{Q,2}(y)|
dyd\theta \widetilde{\psi}_{2^m-j}(t)dt \\
&\lesssim &2^{-(1-\kappa)s}\|\Omega_i\|_{L^{\infty}(S^{d-1})}\|b^l_{Q,2}\|_{L^1(\mathbb{R}^d)}.
\end{eqnarray*}
On the other hand, we have
\begin{eqnarray*}
&&\int_{\mathbb{R}^d}\int_{Q}\int_{Q}|R_{s,a}^{j}(y+t+r\theta,y)-R_{s,a}^j(y'+t+r\theta,y')|
|b^l_{Q,2}(y)| |\widetilde{\psi}_{2^m-j}(t)|dydy'dt \\
&&\quad\lesssim2^{-(1-\iota)s}2^{j}|Q|\|b^l_{Q,2}\|_{L^1(\mathbb{R}^d)},
\end{eqnarray*}
which implies that
$$ \|{\rm V}\|_{L^1(\mathbb{R}^d)}\lesssim2^{-(1-\iota)s}\|\Omega_i\|_{L^{\infty}(S^{d-1})}
 \|b^l_{Q,2}\|_{L^1(\mathbb{R}^d)}.$$
This finish the proof of Lemma \ref{lemma5.3}.
\end{proof}

For $s>3$, let  $\mathfrak{E}^s=\{e_{\nu}^s\}_{\nu\in\Lambda_s}$ be a collection of unit vectors on $S^{d-1}$ such that
\begin{itemize}
\item[\rm (a)] $|e_{\nu}^s-e_{\nu'}^s|>2^{-s\gamma-4}$ when $\nu\not=\nu'$;
\item[\rm (b)] for each $\theta\in S^{d-1}$, there exists an $e^s_{\nu}$ such that $|e^s_{\nu}-\theta|\leq 2^{-s\gamma-4},$
\end{itemize}
where $\gamma\in (0,\,1)$ is a constant.  The set $\mathfrak{E}^s$ can be constructed as in \cite[Section 2]{dinglai}.  Observe that ${\rm card}(\mathfrak{E}^s)\lesssim 2^{s\gamma(d-1)}$.
Let $\zeta$  be a smooth, nonnegative, radial function, such that ${\rm supp}\, \zeta\subset B(0,\,1)$ and $\zeta(t)=1$ for $|t|\leq 1/2$. Set
$$\widetilde{\Gamma}_\nu^s(\xi)=\zeta\Big(2^{s\gamma}\big(\frac{\xi}{|\xi|}-e^s_{\nu}\big)\Big)$$
and$$\Gamma_{\nu}^s(\xi)=\widetilde{\Gamma}^s_{\nu}(\xi)\Big(\sum_{\nu\in\Lambda_s}\widetilde{\Gamma}^s_{\nu}(\xi)\Big)^{-1}.$$ It is easy to verify that $\Gamma^s_{\nu}$ is homogeneous of degree zero, and for all $s$,
$$\sum_{\nu\in\Lambda_s}\Gamma^s_{\nu}(\xi)=1,\,\, \xi\in S^{d-1}.
 $$

Let $\varpi\in C^{\infty}_0(\mathbb{R})$ such that $0\leq \varpi\leq 1$,  ${\rm supp}\, \varpi\subset [-4,\,4]$ and $\varpi(t)\equiv 1$ when $t\in [-2,\,2]$. Define the multiplier operator $G_{\nu}^s$ by
$$\widehat{G_{\nu}^sf}(\xi)=\varpi\big(2^{s\gamma}\langle \xi/|\xi|, e_{\nu}^s\rangle\big)\widehat{f}(\xi),
$$
Let $\mathcal{J}_{m,\,a;\,j,\nu}^{i, s}$ be defined by
$$\mathcal{J}_{m,\,a;\,j,\nu}^{i, s}h(x)=\int_{\mathbb{R}^d}\widetilde{K}_{j,m}^i(x-y)\Gamma^s_{\nu}(x-y)R^j_{s,\,a}(x,\,y)h(y)dy.$$
It is obvious that for each fixed $s\in\mathbb{N}$,
$$\sum_{\nu\in\Lambda_s}\mathcal{J}_{m,\,a;\,j,\nu}^{i, s}h(x)=\mathcal{J}_{m,\,a;\,j}^{i}h(x).$$

\begin{lemma}\label{lemma5.4}
Under the hypothesis of Theorem \ref{dingli1.main}, we have that for each $s\geq 3$,
$$\Big\|\sum_{j}\sum_{\nu}G_{\nu}^s(I-\Psi_{s\kappa -j})\mathcal{J}_{m,\,a;\,j,\nu}^{i, s}B_{j-s,2}^l
\Big\|^2_{L^2(\mathbb{R}^d)}\lesssim \|\Omega_i\|_{L^{\infty}(S^{d-1})}^22^{-s\gamma }
\sum_{Q\in\mathcal{S}}\|b_{Q,2}^l\|_{L^1(\mathbb{R}^d)}.$$
\end{lemma}
\begin{proof}
For   $f\in L^2(\mathbb{R}^d)$, it follows from Cauchy-Schwarz inequality that
\begin{eqnarray*}
&&\Big|\sum_{j}
\sum_{\nu}\int_{\mathbb{R}^d}G_{\nu}^s(I-\Psi_{s\kappa -j})\mathcal{J}_{m,\,a;\,j,\nu}^{i, s}B^l_{j-s,2}(x)f(x)dx\Big|\\
&&\quad=\Big|\int_{\mathbb{R}^d}
\sum_{\nu}G_{\nu}^sf(x)\sum_{j}(I-\Psi_{s\kappa -j})\mathcal{J}_{m,\,a;\,j,\nu}^{i, s}B_{j-s,2}^l(x)dx\Big|\\
&&\quad\leq\Big\|\Big(\sum_{\nu}|G_{\nu}^sf|^2\Big)^{\frac{1}{2}}\Big\|_{L^2(\mathbb{R}^d)}
\Big(\sum_{\nu}\Big\|\sum_{j}(I-\Psi_{s\kappa -j})\mathcal{J}_{m,\,a;\,j,\nu}^{i, s}B_{j-s,2}^l\Big\|^2_{L^2(\mathbb{R}^d)}\Big)^{\frac{1}{2}}.
\end{eqnarray*}
Plancherel's theorem, via the estimate
$$\sup_{\xi\not=0}\sum_{\nu}|\psi(2^{s\gamma}\langle e^s_{\nu},\xi/|\xi|\rangle)|^2\lesssim 2^{s\gamma(d-2)}
$$
(see \cite[inequality (3.1)]{dinglai}), implies that
\begin{eqnarray}\label{equation3.3}\Big\|\Big(\sum_{\nu}|G_{\nu}^sf|^2\Big)^{\frac{1}{2}}\Big\|_{L^2(\mathbb{R}^d)}^2 &=&\sum_{\nu}\int_{\mathbb{R}^d}|\psi(2^{s\gamma}\langle \xi/|\xi|, e_{\nu}^s\rangle)|^2|\widehat{f}(\xi)|^2d\xi\\ &\lesssim&2^{s\gamma(d-2)}\|f\|_{L^2(\mathbb{R}^d)}^2.\nonumber\end{eqnarray}
Recall that ${\rm card}(\mathfrak{E}^s)\lesssim 2^{\gamma s(d-1)}$. It suffices to prove that for each fixed $\nu\in\Lambda_s$,
\begin{eqnarray}\label{equation5.4}
&&\Big\|\sum_{j}(I-\Psi_{s\kappa -j})\mathcal{J}_{m,\,a;\,j,\nu}^{i, s}B^l_{j-s,2}\Big\|^2_{L^2(\mathbb{R}^d)}\\
&&\quad\lesssim  2^{-2s\gamma(d-1)}\|\Omega_i\|^2_{L^{\infty}(S^{d-1})}
\sum_{Q\in\mathcal{S}}\|b_{Q,2}^l\|_{L^1(\mathbb{R}^d)}.\nonumber
\end{eqnarray}

We now prove (\ref{equation5.4}).
For each fixed $j$, $Q\in\mathcal{S}_{j-s}$, and $x\in\mathbb{R}^d$, write
\begin{eqnarray*}
&&\big|	(I-\Psi_{s\kappa -j})\mathcal{J}_{m,\,a;\,j,\nu}^{i, s}b^l_{Q,2}(x)\big|\label{hjk} \\
&&\quad\lesssim  \int_{\mathbb{R}^d}|\widetilde{K}_{j,m}^i(x-y)||\Gamma^j_{\nu}(x-y)|
|R_{s,a}^{j}(x,y)||b^l_{Q,2}(y)|dy \\
			&&\qquad+ \int_{\mathbb{R}^{d}}\int_{\mathbb{R}^{d}} |\widetilde{K}_{j,m}^i(z-y)|
|\psi_{s\kappa-j}(x-z)| |\Gamma^j_{\nu}(z-y)||R_{s,a}^{j}(z,y)|dz|b^l_{Q,2}(y)|dy\\
			&&\quad\lesssim \mathscr{H}_{j,\nu}^s*|b^l_{Q,2}|(x),
\end{eqnarray*}
where $\mathscr{H}_{j,\nu}^{i,s}(y)=2^j\big|\widetilde{K}_{j,m}^i(y)\Gamma^j_{\nu}(y)\big|$.  Write
\begin{eqnarray*}
&&\Big\|\sum_{j}\sum_{Q\in{j-s}}(I-\Psi_{s\kappa -j})\mathcal{J}_{m,\,a;\,j,\nu}^{i, s}b^l_{Q,2}
\Big\|_{L^2(\mathbb{R}^d)}^2\\
&&\quad\lesssim\sum_{j}\sum_{Q\in{S_{j-s}}}\sum_{I\in{S_{j-s}}}\int_{\mathbb{R}^d}
(\mathscr{H}_{j,\nu}^{i,s}*
\mathscr{H}_{j,\nu}^{i,s}*|b_{I,2}^l|)(x)
|b^l_{Q,2}(x)|dx\\
&&\qquad+2\sum_{j}\sum_{Q\in\mathcal{S}_{j-s}}\sum_{u<j}\sum_{I\in\mathcal{S}_{u-s}}
\int_{\mathbb{R}^d}(\mathscr{H}_{j,\nu}^{i,s}*\mathscr{H}_{u,\nu}^{i,s}*|b^l_{I,2}|)(x)|b^l_{Q,2}(x)|dx.
\end{eqnarray*}
Let
$$\mathcal{R}_{j,\nu}^{s}={\{x\in\mathbb{R}^d:\,|\inn{x}{e_v^s}|\leq2^{j+3},|x-\inn{x}{e_{\nu}^s}e_{\nu}^s|
\leq2^{j+3-s\gamma}\}},$$
and $\widetilde{\mathcal{R}}_{j,\nu}^s=\mathcal{R}_{j,\nu}^s+\mathcal{R}_{j,\nu}^s.$ Observe that $\mathcal{R}_{j,\nu}^s$ is contained in a box having one long side of length
$\lesssim2^j$ and $(d-1)$ short sides of length $\lesssim2^{j-s\gamma}$.
For each fixed $j$, $Q\in\mathcal{S}_{j-s}$ and  $x\in{Q}$, we obtain
\begin{eqnarray*}
&&\sum_{u\leq j}\sum_{I\in\mathcal{S}_{u-s}}\big(\mathscr{H}_{j,\nu}^{i,s}*\mathscr{H}_{u,\nu}^{i,s}*|b^l_{I,2}|
\big)(x)\\
&&\quad\leq \|\mathscr{H}_{j,\nu}^{i,s}\|_{L^{\infty}(\mathbb{R}^d)}
\sum_{u\leq j}\|\mathscr{H}_{u,\nu}^{i,s}\|_{L^{1}(\mathbb{R}^d)}\sum_{I\in\mathcal{S}_{u-s}}
\|b_{I,2}^l\|_{L^1(\mathbb{R}^d)}\\
&&\quad\lesssim \|\Omega_i\|_{L^{\infty}(S^{d-1})}^22^{-jd}2^{-s\gamma(d-1)}\sum_{u\leq j}
\sum_{I\in\mathcal{S}_{u-s}}
\int_{x+\widetilde{\mathcal{R}}_{j,\nu}^s}|b_{I,2}^l(y)|dy\\
&&\quad\lesssim \|\Omega_i\|_{L^{\infty}(S^{d-1})}^22^{-jd}2^{-s\gamma(d-1)}\sum_{u\leq j}
\sum_{I\in\mathcal{S}_{u-s}\atop{I\cap \widetilde{\mathcal{R}}_{j,\nu}^s\not =\emptyset }}|I|\\
&&\quad\lesssim  2^{-2s\gamma(d-1)}\|\Omega_i\|_{L^{\infty}(S^{d-1})}^2.
\end{eqnarray*}
This verifies (\ref{equation5.4}) and  completes the proof of Lemma \ref{lemma5.4}.
\end{proof}

\begin{lemma}\label{lemma5.5} For each fixed $j\in\mathbb{Z}$, $s\in\mathbb{N}$ and $Q\in\mathcal{S}_{j-s}$, there exists a positive constant $\varrho>0$ such that
$$\sum_{\nu}  \| (I-{G_{\nu}^s)}(I-\Psi_{s\kappa -j})\mathcal{J}_{m,\,a;\,j,\nu}^{i, s}b^l_{Q,2}
\|_{L^1(\mathbb{R}^d)}\lesssim  2^{-\varrho s}  \|\Omega_i\|_{L^{\infty}(\mathbb{S}^{d-1})} \|b^l_{Q,2}\|_{L^1(\mathbb{R}^d)}.$$
\end{lemma}
\begin{proof}
Let $\alpha$ be a radial $C^{\infty}$ function such that $\alpha(\xi)=1$ for $|\xi|\leq1$,
$\alpha(\xi)=0$ for $|\xi|\geq2$ and $0\leq\alpha(\xi)\leq1$ for all $\xi\in\mathbb{R}^d.$
Set $\beta_{k}(\xi)=\alpha(2^k\xi)-\alpha(2^{k+1}\xi).$ Let  $\widetilde\beta$ be a radial $C^{\infty}$
function such that $\widetilde\beta(\xi)=1$ for $\frac{1}{2}\leq|\xi|\leq2,$
${\rm supp}\,\widetilde\beta\in[1/4,4]$ and $0\leq\widetilde\beta\leq1$ for all
$\xi\in\mathbb{R}^d$. Set $\widetilde\beta_k(\xi)=\widetilde\beta(2^k\xi).$  Define  multipliers
operator  $\Lambda_{k}$ and $\widetilde\Lambda_{k}$ by
$$\widehat{\Lambda_{k}f}(\xi)=\beta_{k}(\xi)\widehat f(\xi),\,\,\widehat{\widetilde\Lambda_{k}f}(\xi)
=\widetilde\beta_{k}(\xi)\widehat f(\xi).$$
Observe that
$\beta_k=\widetilde\beta_{k}\beta_{k},$  and so $\Lambda_{k}=\widetilde\Lambda_{k}\Lambda_{k}$. It then follows that
\begin{eqnarray}
(I-G_{\nu}^s)\mathcal{J}_{m,\,a;\,j,\nu}^{i, s}=\sum_{k}(I-G_{\nu}^s)\Lambda_{k}\mathcal{J}_{m,\,a;\,j,\nu}^{i, s},
\end{eqnarray}and
\begin{eqnarray*}
&&\|(I-G_{\nu}^s)(I-\Psi_{s\kappa -j})\Lambda_{k}\mathcal{J}_{m,\,a;\,j,\nu}^{i, s}b^l_{Q,2}
\|_{L^1(\mathbb{R}^d)}\\
&&\quad\leq\|(I-\Psi_{s\kappa -j})\widetilde\Lambda_{k}(I-G_{\nu}^s)\Lambda_{k}
\mathcal{J}_{m,\,a;\,j,\nu}^{i, s}
b^l_{Q,2}\|_{
L^1(\mathbb{R}^d)}\\
&&\quad\leq\|(I-\Psi_{s\kappa -j})\widetilde\Lambda_{k}\|_{L^1(\mathbb{R}^d)\rightarrow L^1(\mathbb{R}^d)}
\|(I-G_{\nu}^s)\Lambda_k\mathcal{J}_{m,\,a;\,j,\nu}^{i, s}b^l_{Q,2}\|_{L^1(\mathbb{R}^d)}.
\end{eqnarray*}
Now let $\widetilde{L}(x,\,y)$ be the kernel of the operator $(I-G_{\nu}^s)\Lambda_k
\mathcal{J}_{m,\,a;\,j,\nu}^{i,s}$,
that is,
\begin{eqnarray*}
(I-G_\nu^s)\Lambda_{k}\mathcal{J}_{m,\,a;\,j,\nu}^{i,s}b^l_{Q,2}(x)
=\int_{\mathbb{R}^d}\widetilde{L}(x,y)b^l_{Q,2}(y)dy.
\end{eqnarray*}
It then follows that
$$\|(I-G_{\nu}^s)\Lambda_k \mathcal{J}_{m,\,a;\,j,\nu}^{i,s}b^l_{Q,2} \|_{L^1(\mathbb{R}^d)}\leq\int_{Q}
\|\widetilde{L}(\cdot,y)
\|_{L^1(\mathbb{R}^d)}b^l_{Q,2}(y)dy.$$
Repeating the proof of Lemma 4.2 in \cite{dinglai}, we  know that there exists a constant $\upsilon>0$ such that
$$\|\widetilde{L}(\cdot,y)\|_{L^1(\mathbb{R}^d)}\lesssim 2^{\iota s-s\gamma(d-1)-j+k+s\gamma \upsilon}\|\Omega_i\|_{L^{\infty}(S^{d-1})}.$$
Note that
$$\|(I-\Psi_{s\kappa -j})\widetilde\Lambda_{k}\|_{L^1(\mathbb{R}^d)\rightarrow L^1(\mathbb{R}^d)}
\leq\|F^{-1}(\widetilde \beta_{k})-\psi_{s\kappa-j}*F^{-1}(\widetilde\beta_{k})\|_{L^1(\mathbb{R}^d)}
\lesssim1.$$ It then follows that
\begin{eqnarray}\label{eq5.2secondtolast}&&\|(I-G_{\nu}^s)(I-\Psi_{s\kappa -j})\Lambda_{k}
\mathcal{J}_{m,\,a;\,j,\nu}^{i,s}b^l_{Q,2}\|_{L^1(\mathbb{R}^d)}\\
&&\quad\lesssim2^{\iota s-s\gamma(d-1)-j+k+s\gamma \upsilon}
\|\Omega_i\|_{L^{\infty}(S^{d-1})}\|b^l_{Q,2}\|_{L^1(\mathbb{R}^d)}.\nonumber
\end{eqnarray}
On the other hand, we can write
\begin{eqnarray*}
&&\|(I-G_{\nu}^s)(I-\Psi_{s\kappa -j})\Lambda_{k}\mathcal{J}_{m,\,a;\,j,\nu}^{i,s}b^l_{Q,2}
\|_{L^1(\mathbb{R}^d)}\\
&&\quad\leq \|(I-\Psi_{s\kappa -j})\widetilde\Lambda_{k}\|_{L^1(\mathbb{R}^d)\rightarrow L^1(\mathbb{R}^d)}\\
&&\qquad\times \|(I-G_{\nu}^s)\Lambda_k\|_{L^1(\mathbb{R}^d)\rightarrow L^1(\mathbb{R}^d)}
\|\mathcal{J}_{m,\,a;\,j,\nu}^{i,s}b^l_{Q,2}\|_{L^1(\mathbb{R}^d)}.
\end{eqnarray*}
Let $W_{k,j,\kappa}^s$ be the kernel of the convolution operator $(I-\Psi_{s\kappa -j})\widetilde\Lambda_{k}$. The mean value formula now tells us that
\begin{eqnarray*}\int_{\mathbb{R}^d}|W_{k,j,\kappa}^s(y)|dy&\leq &
\int_{\mathbb{R}^d}\int_{\mathbb{R}^d}|F^{-1}\widetilde\beta_{k}(y)-F^{-1}\widetilde\beta_{k}(y-z)
|\psi_{s\kappa-j}(z)dzdy\\
&\lesssim &2^{j-s\kappa-k}.\nonumber
\end{eqnarray*}
Obviously, we have that
$$\|(I-G_{\nu}^s)\Lambda_k\|_{L^1(\mathbb{R}^d)}\lesssim1 $$ and
$$\|\mathcal{J}_{m,\,a;\,j,\nu}^{i, s}b^l_{Q,2}\|_{L^1(\mathbb{R}^d)}\lesssim 2^{-s\gamma(d-1)}\|\Omega_i\|_{L^{\infty}(S^{d-1})}
\|b^l_{Q,2}\|_{L^1(\mathbb{R}^d)}.
$$
Therefore,
\begin{eqnarray}\label{eq5.2last}
&&\|(I-G_{\nu}^s)(I-\Psi_{s\kappa -j})\Lambda_{k}\mathcal{J}_{m,\,a;\,j,\nu}^{i, s}b^l_{Q,2}
\|_{L^1(\mathbb{R}^d)}
\lesssim2^{-s\gamma(d-1)+j-k-s\kappa}\|b^l_{Q,2}\|_{L^1(\mathbb{R}^d)}.
\end{eqnarray}

Let $m=j-\lfloor s\varepsilon_{0}\rfloor$, with $0<\varepsilon_{0}<1 $ a constant. Recall that ${\rm card}(\mathfrak{E}^s)
\lesssim2^{s\gamma(d-1)}.$ Combining the inequalities (\ref{eq5.2secondtolast}) and (\ref{eq5.2last})   leads to that
\begin{eqnarray*}
&&\sum_{\nu}\|(I-G_{\nu}^s)(I-\Psi_{s\kappa -j})\mathcal{J}_{m,\,a;\,j,\nu}^{i, s}b^l_{Q,2}\|_{L^1(\mathbb{R}^d)}\\
&&\quad\leq(\sum_{\nu}\sum_{k<m}+\sum_{\nu}\sum_{k\geq m})\|(I-\Psi_{s\kappa -j})(I-G_{\nu}^s)
\Lambda_{k}\mathcal{J}_{m,\,a;\,j,\nu}^{i, s}b^l_{Q,2}\|_{L^1(\mathbb{R}^d)}\\
&&\quad\lesssim(2^{s\varrho_*}+2^{s\varrho^*})\|\Omega_i\|_{L^{\infty}(S^{d-1})}
\|b^l_{Q,2}\|_{L^1(\mathbb{R}^d)},
\end{eqnarray*}
where $\varrho_*=(\iota-\varepsilon_{0}+\gamma\upsilon)$ and $\varrho^*=-\kappa+\varepsilon_0.$
We   choose $0<\iota<\gamma<\varepsilon_0<\kappa<1$ such that $\max{\{\varrho_*,\varrho^*\}}<0$.
Let $\varrho=-\max\{\varrho_*,\varrho^*\}$. The conclusion of Lemma \ref{lemma5.5} now follows directly.
\end{proof}
\subsection{Proof of inequality (\ref{eq3.last})} At first, we choose the constant $N_1>10$.
Let $$W_1(x)=\sum_{i=0}^{\infty}\sum_{l=1}^{\infty}\sum_{m=1}^{l+i^*}
\sum_{s>N_1(l+i)}\sum_{r=0}^{s2^{l+2}-1}M\Big(\sum_{j}
\big|\big(\mathcal{H}_{m,a;\,j}^i-\mathcal{J}_{m,\,a;\,j}^i)B^{l}_{j-s}\big|\Big)(x),
$$
$$W_2(x)=\sum_{i=0}^{\infty}\sum_{l=1}^{\infty}\sum_{m=1}^{l+i^*}
\sum_{s>N_1(l+i)}\sum_{r=0}^{s2^{l+2}-1}M\Big(\sum_{j=r(\hbox{mod} \ s2^{l+2})}
\mathcal{J}_{m,\,a;\,j}^iB_{j-s, 2}^l\Big)(x),
$$
and
$$W_3(x)=\sum_{i=0}^{\infty}2^i\sum_{l=1}^{\infty}(l+i)
\sum_{s> N_1(l+i)}s2^{l+2}2^{-s}
\sum_{j\in\mathbb{Z}}\sum_{Q\in\mathcal{S}_{j-s}}v_{j}*|b^l_{Q,2}|(x).$$
It then follows from Lemma \ref{lem4.2} that for some positive constant $C_1$,
$${\rm D}_{52}b(x)\leq \frac{C_1}{64}\Big(W_1(x)+W_2(x)+W_3(x)
+Mb(x)\Big).$$
Take $N_1=10\lfloor \iota^{-1}\rfloor+1$. A trivial computation  leads to that
\begin{eqnarray*}
&&| \{x\in\mathbb{R}^d:\, W_3(x)>C_1^{-1} \} | \\
&&\quad\lesssim \sum_{i=0}^{\infty}2^i\sum_{l=1}^{\infty}(l+i)
\sum_{s> N_1(l+i)}s2^{l+2}2^{-s}
\sum_{j\in\mathbb{Z}}\sum_{Q\in\mathcal{S}_{j-s}}\|v_{j}*|b^l_{Q,2}|\|_{L^1(\mathbb{R}^d)}\\
&&\quad\lesssim \sum_{i=0}^{\infty}2^i\sum_{l=1}^{\infty}(l+i)2^{l}
\sum_{s=N_1(l+i)}^{\infty}s2^{-s}\|f\|_{L^1(\mathbb{R}^d)}\lesssim \|f\|_{L^1(\mathbb{R}^d)}.
\end{eqnarray*}
Let $$c_2=\big(\sum_{i=1}^{\infty}i^{-2}\big)^5,\,\,c(i,l,s,m,r)=c_2^{-1}\big((i+1) l sm(r+1)\big)^{-2}.$$
It follows  from the pigeonhole principle,  the estimate (\ref{eq4.approximation})  and the weak type $(1,\,1)$ estimate of $M$, that
\begin{eqnarray*}
&&| \{x\in\mathbb{R}^d:\, W_1(x)>C_1^{-1} \} | \\
&&\quad\lesssim \sum_{i=0}^{\infty}\sum_{l=1}^{\infty}\sum_{m=1}^{l+i^*}
\sum_{s=N_1(l+i)}^{\infty}\sum_{r=0}^{s2^{l+2}-1}\Big|\Big\{x\in\mathbb{R}^d:\\
&&\qquad M\Big(\sum_{j=r(\hbox{mod} \ s2^{l+2})}
\big|\big(\mathcal{H}_{m,a;j}^i-\mathcal{J}_{m,a;j}^i)B_{j-s}\big|\Big)(x)>
c(i,l,s,m,r)C_1^{-1}\Big\}\Big |\\
&&\quad\lesssim \sum_{i=0}^{\infty}i^22^i\sum_{l=1}^{\infty}(l+i)^2l^22^{2l}
\sum_{s=N_1(l+i)}^{\infty} s^{4} 2^{- \iota s}\|f\|_{L^1(\mathbb{R}^d)}\lesssim
\|f\|_{L^1(\mathbb{R}^d)}.
\end{eqnarray*}
Finally, it follows from Lemma \ref{lemma5.3}-Lemma \ref{lemma5.5} that
\begin{eqnarray*}
&&| \{x\in\mathbb{R}^d:\, W_2(x)>C_1^{-1} \} |\lesssim \|f\|_{L^1(\mathbb{R}^d)}.
\end{eqnarray*}
Note that
$$| \{x\in\mathbb{R}^d:\, Mb(x)>C_1^{-1} \} |\lesssim \|f\|_{L^1(\mathbb{R}^d)}.$$
Combining the estimates for $W_1$, $W_2$ and $W_3$ then leads to inequality (\ref{eq3.last}).\qed

\begin{remark} For bounded function $f$ with compact support, let $\mathcal{S}=\{Q\}$ be the cubes in the proof (\ref{budengshi3.1}). Let
$$g^*(x)=f(x)\chi_{\mathbb{R}^d\backslash \cup_{Q\in\mathcal{S}}Q}(x)+f(x)\sum_{Q\in\mathcal{S}}\chi_{Q(|f|\leq 2^{2c_1})}(x),
$$
$$b_{Q,\,1}^{*,l}(x)=\frac{1}{|Q|}\int_{Q(2^{c_12^{2^{l-1}}}<|f|\leq 2^{c_12^{2^l}})}f(y)dy\chi_{Q}(x),$$
and
$$b_{Q,2}^{*,l}(x)=f(x)\chi_{Q(2^{c_12^{2^{l-1}}}<|f|\leq 2^{c_12^{2^l}})}(x)-b_{Q,1}^{*,l}(x).$$
Let $i^{**}=\lfloor\log \log \big(\frac{2}{\tau}({\rm e}^2+i)\big)\rfloor+1$. Observe that
$$\sum_{m=1}^{i^{**}}2^{2^m}\le \sum_{j=1}^{2^{2^{**}}}2^{j}\leq 22^{2^{i^{**}}}\lesssim i^2.$$
Repeating the proof of Theorem \ref{dingli1.main}, but replacing $i^*$, $U_{l,\,a,j}$ and $U_{l,\,a,j}^i$ by $i^{**}$,  $U_{2^l,a,j}$ and $U^i_{2^l,a,j}$ respectively, we can verify the following conclusion.
\begin{theorem}\label{dingli1.mainre}
Let  $\Omega$ be homogeneous of degree zero, satisfy the vanishing condition (\ref{equation1.1}).
Let $a$ be a   function in $\mathbb{R}^d$ such that $\nabla a\in L^{\infty}(\mathbb{R}^d)$.
Suppose that $\Omega\in L(\log L)^2(S^{d-1})$.
Then for all $\lambda>0$, $$
\big|\{x\in\mathbb{R}^d:\,|T_{\Omega,\,a}^*f(x)|>\lambda\}\big|
\lesssim \int_{\mathbb{R}^d}\Phi_3\big(\frac{|f(x)|}{\lambda}\big)dx,
$$
where $\Phi_3(t)=t\log\log\log({\rm e^{e^2}}+t)$ for $t\geq 0$.
\end{theorem}
\end{remark}

\end{document}